\documentclass[preprint,1p]{elsarticle}
\usepackage{a4wide}
\usepackage{amsmath}
\usepackage{amssymb}
\usepackage{graphicx}
\usepackage{color}
\usepackage{epic}
\usepackage{tikz}                                                                                       
\usetikzlibrary{patterns}  
\usetikzlibrary{arrows}
\usepackage[labelfont=bf,skip=2pt]{caption}
\usepackage[bookmarks=false]{hyperref}
\usepackage[hang]{footmisc}

\usepackage{lipsum}
\makeatletter
\def\ps@pprintTitle{%
 \let\@oddhead\@empty
 \let\@evenhead\@empty
 \def\@oddfoot{}%
 \let\@evenfoot\@oddfoot}
\makeatother

\makeatletter
\def\printFirstPageNotes{%
  \iflongmktitle
   \let\columnwidth=\textwidth\fi
  \ifx\@tnotes\@empty\else\@tnotes\fi
  \ifx\@nonumnotes\@empty\else\@nonumnotes\fi
  \ifx\@cornotes\@empty\else\@cornotes\fi
  \ifx\@elseads\@empty\relax\else
   \let\thefootnote\relax
   \footnotetext{\ifnum\theead=1\relax
      \textit{E-Mail:\space}\else
      \textit{E-Mail:\space}\fi
     \@elseads}\fi
  \ifx\@elsuads\@empty\relax\else
   \let\thefootnote\relax
   \footnotetext{\textit{URL:\space}%
     \@elsuads}\fi
  \ifx\@fnotes\@empty\else\@fnotes\fi
  \iflongmktitle\if@twocolumn
   \let\columnwidth=\Columnwidth\fi\fi
}
\makeatother

\newdefinition{definition}[equation]{Definition}
\newdefinition{defineCorrollary}[equation]{Definition/Conclusion}
\newtheorem{theorem}[equation]{Theorem}
\newtheorem{lemma}[equation]{Lemma}
\newdefinition{remark}[equation]{Remark}
\newdefinition{exmp}[equation]{Example}
\newtheorem{corollary}[equation]{Conclusion}
\newdefinition{algorithm}[equation]{Algorithm}

\newtheorem{assertion}[equation]{Assertion}
\newproof{proof}{Proof}
\newproof{prooftheorem}{Proof of Theorem \thetheorem}

\renewenvironment{proof}{{\bfseries Proof.}}{}

\renewcommand{\theequation}{\arabic{section}.\arabic{equation}}

\begin{document}

\begin{center}
{\Large{}Computing best discrete least-squares approximations by first-degree splines with free knots\rule{0pt}{20pt}}
\end{center}

\vspace{0.2cm}
\begin{center}\setlength{\footnotemargin}{0.5em}
{\large Ludwig J. Cromme\footnote{Numerische und Angewandte Mathematik, BTU Cottbus-Senftenberg, Platz der Deutschen Einheit 1, D-03046 Cottbus, Germany, E-mail: Ludwig.Cromme@b-tu.de\rule[-8pt]{0pt}{5pt}}, Jens Kunath\footnote{Numerische und Angewandte Mathematik, BTU Cottbus-Senftenberg, Platz der Deutschen Einheit 1, D-03046 Cottbus, Germany\rule[-8pt]{0pt}{5pt}} and Andreas Krebs\footnote{Aerodynamik und Str\"{o}mungslehre , BTU Cottbus-Senftenberg, Siemens-Halske-Ring 14, D-03046 Cottbus, Germany}}
\end{center}



\vspace{0.2cm}
\centerline{\hrulefill}
\vspace{-0.1cm}
\begin{flushleft}
\textbf{Abstract}
\end{flushleft}

\vspace{-0.2cm}\noindent
We present an algorithm to compute best least-squares approximations of discrete real-valued functions by first-degree splines (broken lines) with free knots. We demonstrate that the algorithm delivers after a finite number of steps a (global) best approximation. The analysis is complemented by remarks on programming and by a number of numerical examples including applications from medicine (MBC, MIC).

\vspace{0.25cm}
\noindent\textit{Key words}: splines, first-degree splines, broken lines, splines with free knots, least-squares approximation, best approximation, minimal bactericidal concentration (MBC), minimal inhibitory concentration (MIC)

\centerline{\hrulefill}

\section{Introduction}
\label{sec:intro}
Splines of a real variable with free knots form a highly nonlinear approximation family with varying dimension of the boundary (Cromme \cite{cromme01}). Even just local best approximations are therefore not easily computed (Cromme \cite{cromme02}) and globally convergent numerical methods are not known.

But for the special case of first-degree splines (broken lines) with free knots a globally convergent numerical method for the approximation of discrete data in the least-squares sense could be developed and is presented in the following sections. 

\textit{The basic idea of the method is to reduce the complexity of the original problem: Instead of solving the original nonlinear problem on the whole data set a number of just linear problems on subsets ("segments") is treated and assembled to a solution of the original problem.}

The underlying idea is sketched in Fig. \ref{fig1}: In this figure the best $L_2$-approximation $s^{\ast}$ from the set of first-degree splines with $5$ free knots $t_1,\ldots,t_5$ to discrete data $(x_i,f_i)$, $i=0,1,\ldots,18$, is plotted. Knots can be located between two data abscissae (such knots will be called interior knots in the following, in our example $t_1$ und $t_4$ are interior knots) or coincide with data abscissae (called data knots, in our example $t_2=x_4, t_3 = x_{10}$, and $t_5 = x_{15}$ are data knots).

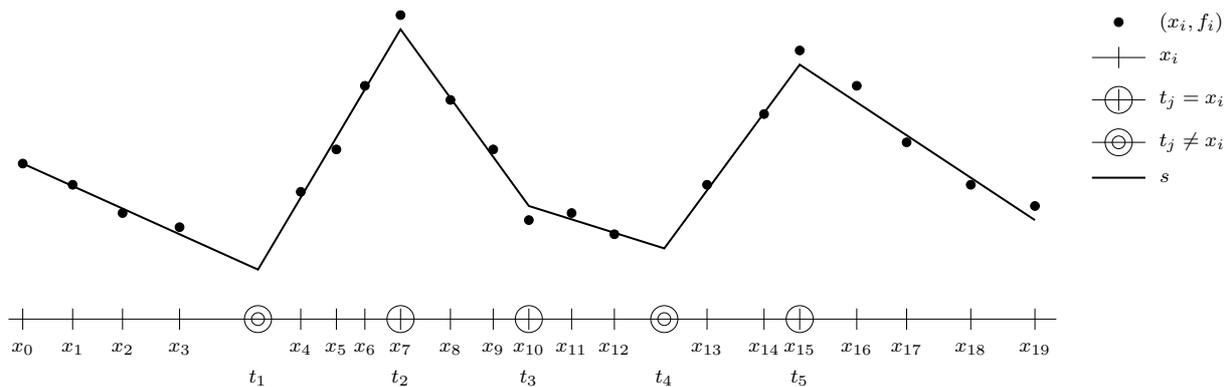
\begin{figure}[h]
\begin{center}
\resizebox{1.0\textwidth}{!}{
\begin{tikzpicture}
\draw (-0.2,0)--(14.5,0);
\foreach \x/\y in {0/0,0.7/1,1.4/2,2.2/3,3.9/4,4.4/5,4.8/6,5.3/7,6/8,6.6/9,7.1/10,7.7/11,8.3/12,9.6/13,10.4/14,10.9/15,11.7/16,12.4/17,13.3/18,14.2/19} \draw (\x,0.15)--(\x,-0.15) node[below=2pt] {\footnotesize $x_{\y}$};
\draw (3.3,0) circle (0.09cm);
\draw (3.3,0) circle (0.19cm) node[below=0.6cm] {\footnotesize $t_1$};
\draw (5.3,0) circle (0.19cm) node[below=0.6cm] {\footnotesize $t_2$};
\draw (7.1,0) circle (0.19cm) node[below=0.6cm] {\footnotesize $t_3$};
\draw (9,0) circle (0.09cm);
\draw (9,0) circle (0.19cm) node[below=0.6cm] {\footnotesize $t_4$};
\draw (10.9,0) circle (0.19cm) node[below=0.6cm] {\footnotesize $t_5$};

\draw[thick,smooth] (0,2.2)--(3.3,0.7)--(5.3,4.1)--(7.1,1.6)--(9,1)--(10.9,3.6)--(14.2,1.4);

\foreach \x/\y in {0/2.2,0.7/1.9,1.4/1.5,2.2/1.3,3.9/1.8,4.4/2.4,4.8/3.3,5.3/4.3,6/3.1,6.6/2.4,7.1/1.4,7.7/1.5,8.3/1.2,9.6/1.9,10.4/2.9,10.9/3.8,11.7/3.3,12.4/2.5,13.3/1.9,14.2/1.6} \draw[fill] (\x,\y) circle (1.75pt);

\draw[fill](15.375,4.2) circle (1.75pt);
\draw (15.75,4.2) node[right=2pt] {\footnotesize $(x_i,f_i)$};
\draw (15,3.7)--(15.75,3.7) node[right=2pt] {\footnotesize $x_i$};
\draw (15.375,3.85)--(15.375,3.55);
\draw (15,3.1)--(15.75,3.1) node[right=2pt] {\footnotesize $t_j = x_i$};
\draw (15.375,3.25)--(15.375,2.95);
\draw (15.375,3.1) circle (0.19cm);
\draw (15,2.5)--(15.75,2.5) node[right=2pt] {\footnotesize $t_j\neq x_i$};
\draw (15.375,2.5) circle (0.09cm);
\draw (15.375,2.5) circle (0.19cm);
\draw[thick] (15,2.0)--(15.75,2.0) node[right=2pt] {\footnotesize $s$};

\end{tikzpicture}
} 
\end{center}
\caption{A first-degree spline (broken line) $s^\ast \in S^1_5[a,b]$ approximating the $\mu+2=20$ data $(x_0,f_0),\ldots,(x_{19},f_{19})$. $s^\ast$ has $k=5$ simple knots $t_1,\ldots,t_5$ and the boundary knots $t_0 := a := x_0$ and $t_6 := b := x_{19}$.}
\label{fig1}
\end{figure}

Basic is the observation that $s^{\ast}$ is not only best approximation on the whole data set, but also on subsets of the data called segments; for a formal introduction of notations see below. For example $\left.s^{\ast}\right|_{\left[x_4,x_{12}\right]}$ is a best approximation from $S^1\left(x_7,x_{10}\right)$, the first-degree splines with fixed knots $x_7$ and $x_{10}$, to the data $\left(x_4,f_4\right),\ldots,\left(x_{12},f_{12}\right)$. For if a better approximation $\tilde{s}$ with fixed knots $x_7$ and $x_{10}$ to the data $\left(x_4,f_4\right),\ldots,\left(x_{12},f_{12}\right)$ existed, $s^{\ast}$ could be modified in the considered segment in the direction of $\tilde{s}$. The interior knots $t_1$ and $t_4$ would be slightly moved in this process, the other knots would remain unchanged. The approximation error of the thus modified spline would be smaller than the approximation error of $s^{\ast}$ - contradicting the optimality of $s^{\ast}$ as best approximation.

So if these segments ranging from one interior knot to the next interior knot were known, we could determine $s^{\ast}$ by solving just \textit{\mbox{linear (!)}} approximation problems with fixed knots segment for segment. But since the segments of a best approximation are not known in advance our algorithm examines systematically possible segmentations of the data abscissae to see if they might lead to a best approximation.

The implementation of this idea raises a number of questions which will be adressed after formal presentation of the algorithm in the following section.

Let us summarize the most important notations used: Let $\mathbb{P}_m$ denote the real-valued polynomials of degree smaller or equal $m$ and $[a,b]$ a real interval $a < b$. For $a =: t_0 < t_1 < \ldots, < t_k < t_{k+1} := b$ and $m, k \in \mathbb{N}$ let 
\begin{equation*}
S^m(t_1,\ldots,t_k) := \left\{ \left. s \in C^{m-1}[a,b] \;\;\right|\;\; s_j|_{(t_j,t_{j+1})} \in \mathbb{P}_m \;,\; j = 0,1,\ldots,k \right\}
\end{equation*} 
be the set of splines of degree $m$ with $k$ fixed (simple) knots $t_1,\ldots,t_k$. The splines of degree $m \in \mathbb{N}$ with at most $k \in \mathbb{N}$ free (simple) knots are the set
\begin{eqnarray*}
S^m_k[a,b] &:=& \Big\{ s\in C^{m-1}[a,b] \;\; | \;\; \text{there are points} \\
&& a =: t_0 < t_1 < \ldots < t_k < t_{k+1} := b, \;\; \\ 
&& \text{such that} \;\; s|_{\left(t_i,t_{i+1}\right)} \in \mathbb{P}_m \; \text{for} \;\; i=0,1,\ldots,k  \Big\} \quad.
\end{eqnarray*} 
$t_i$ is called an \textit{improper} (or \textit{inactive}) knot of a spline $s\in S^m_k[a,b]$ if $s$ is  $m$-times differentiable in $t_i$, otherwise the knot is called \textit{proper} (or \textit{active}) knot.

Let $\mu+2$ real abscissae $a =: x_0 < x_1 < \ldots < x_{\mu+1} := b$ and function values $f_0,\ldots,f_{\mu+1} \in \mathbb{R}$ of a function $f:[a,b] \rightarrow \mathbb{R}$ be given. We define vectors $X := \left(x_0,\ldots,x_{\mu+1}\right)^t$, \mbox{$F := \left(f_0,\ldots,f_{\mu+1}\right)^t$}, and for $g: \left\{x_0,\ldots,x_{\mu+1}\right\} \rightarrow \mathbb{R}$ the vector $g(X) := \left(g(x_0),\ldots,g(x_{\mu+1})\right)^t$. With the Euclidian norm $\left\|\cdot\right\|_2$ our problem can now be stated:
\begin{definition}\label{eq:approxprobelmdef}
(\textbf{Approximation problem}) Find $s^{\ast} \in S^1_k[a,b]$ with 
\begin{equation*}
\left\|F-s^{\ast}\right\|_{2,X} := \left\|F-s^{\ast}(X)\right\|_2 = \inf_{s \in S^1_k[a,b]} \left\|F-s(X)\right\|_2 \quad.
\end{equation*}
\end{definition}
In the following sections we generally assume $k \geq 1$ and $\mu \geq k+1 \geq 2$, since otherwise the problem would be a purely polynomial approximation problem or the data could be reproduced exactly in $S_k^1[a,b]$.

\vspace{0.2cm}
\setcounter{equation}{0}
\section{The basic algorithm and convergence theorem}
\label{sec:mainpart}
A spline $s \in S^1_k[a,b]$ partitions the data abscissae $x_0 < x_1 < \ldots < x_{\mu+1}$ naturally into segments each ranging from an interior knot (i.e. a knot situated between two data abscissae) to the next interior knot. Within a segment there can be more data knots, i.e. knots coinciding with data abscissae. 

Such segmentations are basic for the numerical procedure presented below. In defining the segments it is advantageous to code the position of knots in a "position vector", i.e. the information wether a knot coincides with an abscissa and with which or between which abscissae a knot is located. To start we define position vectors for arbitrary first-degree splines. As will be seen later, the algorithm needs to deal only with "regular" position vectors, see Defintion \ref{eq:defRegularPositionsvektor}.

\begin{definition}\label{eq:defPositionsvektor}
\begin{itemize}
	\item[(a)] A position vector $p$ of length $k$ to the data $x_0,\ldots,x_{\mu+1}$ is a vector from $\mathbb{N}_0^k$ with
\begin{equation*}
0 \;\leq\; p(j) \;\leq\; p(j+1) \;\leq\; 2\mu\;,\;\; j=1,2,\ldots,k-1\;,
\end{equation*}
\begin{equation*}
p(j) \;=\; p(j+1) \quad\Rightarrow\quad \text{$p(j)$ is even},
\end{equation*}
where $\mathbb{N}_0 := \mathbb{N}\cup\{0\}$ and $0$ is considered even.
	\item[(b)] For $s \in S^1_k[a,b]$, $a =: x_0 < x_1 < \ldots < x_{\mu+1} := b$ and $x_0 < t_1 < t_2 < \ldots < t_k < x_{\mu+1}$ the position vector $p \in \mathbb{N}_0^k$ of $s$ to the data $x_0,\ldots,x_{\mu+1}$ is defined as
\begin{equation*}
p(j) \;:=\; \left\{
\begin{array}{lcl}
2i-1 & , & \text{if} \;\; t_j = x_i \\
2i & , & \text{if} \;\; t_j \in \left(x_i,x_{i+1}\right) \quad.
\end{array}
\right.
\end{equation*}
\end{itemize}
\end{definition}

\vspace {0.1cm}
\begin{corollary}
\begin{itemize}
	\item[(a)] The position vector for $s^{\ast}$ from Fig. \ref{fig1} is the vector $p = (6, 13, 19, 24, 29)$.
	\item[(b)] A position vector of $s$ allows us to infer back the position of the knots of $s$: $t_j = x_{\frac{p(j)+1}{2}}$, if $p(j)$ is odd and $t_j\in\left(x_{\frac{p(j)}{2}},x_{\frac{p(j)}{2}+1}\right)$, if $p(j)$ is even. By Definition \ref{eq:defPositionsvektor}{} (b) the position of knots is mapped one-to-one onto the set of position vectors.  
	\item[(c)] Successive components $p(j)$ and $p(j+1)$ of a position vector $p$ for $s\in S_k^1[a,b]$ are equal iff $t_j$ and $t_{j+1}$ lie in the same interval $\big(x_i,x_{i+1}\big)$:
\begin{equation*}
p(j) \;=\; p(j+1) \quad\Leftrightarrow\quad x_{\frac{p(j)}{2}} < t_j < t_{j+1} < x_{\frac{p(j)}{2}+1} \quad.
\end{equation*}
  \item[(d)] For every position vector $p$ according to Definition \ref{eq:defPositionsvektor} (a) exists a spline $s\in S_k^1[a,b]$ with $p$ as position vector. To be more precice, there exists a whole class of splines all of which have $p$ as corresponding position vector.
\end{itemize}
\end{corollary}
\begin{proof}
We confine ourselves to the proof of (d). Let $p$ be a position vector according to Definition \ref{eq:defPositionsvektor} (a). For odd $p(j)$ set $t_j := x_{\frac{p(j)+1}{2}}$. For even $p(j)$ without less of generality let $p(j)$ to $p(j+q)$ be exactly the components of $p$ identical with $p(j)$, that is, $p(j) = p(j+1) = \ldots = p(j+q)$ and $p(j-1) < p(j)$, if $j \geq 2$ and $p(j+q) < p(j+q+1)$, if $j+q < k$. Choose $t_j,\ldots,t_{j+q}$ such that
\begin{equation*}
x_{\frac{p(j)}{2}} \;<\; t_j \;<\; t_{j+1} \;<\; \ldots \;<\; t_{j+q} \;<\; x_{\frac{p(j)}{2}+1}
\end{equation*}
holds. This implies $t_1 < t_2 < \ldots < t_k$ and each spline from $S_k^1[a,b]$ with these knots has position vector $p$. \hfill $\square$
\end{proof}

\vspace{0.2cm}
\noindent{}The position of the knots (coded by the position vector) induces a segmentation of the data. Each segment ranges from one interior knot to the next, more precisely:
\begin{defineCorrollary}\label{definitonSegmente}
Let $p\in\mathbb{N}_0^k$ be a position vector for $s\in S^1_k[a,b]$ according to Definition \ref{eq:defPositionsvektor}:
\begin{itemize}
	\item \textbf{Case 1:} All $p(j)$ are odd ($j=1,2,\ldots,k$). Then $s$ has no interior knots, all knots coalesce with date abscissae and we have
\begin{equation*}
t_j := x_{\frac{p(j)+1}{2}} \quad \text{for $j=1,2,\ldots,k$} \quad.
\end{equation*}	
There is only one segment $S_0 := \left\{ x_0, x_1, \ldots, x_{\mu+1} \right\}$. 
	
	\item \textbf{Case 2:} At least one component $p(j)$ of $p$ is even. Let $r\in\{1,\ldots,k\}$ be chosen maximal such that $p(j_l)$ is even for $l=1,\ldots,r$ and suitable $j_l$ with $1 \leq j_1 < j_2 < \ldots < j_r \leq k$. The knots $t_{j_1},\ldots,t_{j_r}$ are exactly the interior knots of $s$. Set
\begin{equation*}
i_l \;:=\; \frac{p\left(j_l\right)}{2} \;\;\;,\;\;\; l = 1,2,\ldots,r \quad.
\end{equation*} 
Let $r \geq 2$, $l \in \{1,\ldots,r-1\}$ hold. In case $i_l = i_{l+1}$, set $S_l := \emptyset$. In case $i_l+1 \leq i_{l+1}$ set 
\begin{equation*}
S_l \;:=\; \left\{x_{i_l+1}, x_{i_l+2},\ldots,x_{i_{l+1}}\right\} \quad.
\end{equation*}
If $i_l+1 \leq i_{l+1}$ and in addition $q_l := j_{l+1} - j_l - 1 > 0$ holds then
\begin{equation*}
t_{j_l + i} \;=\; x_{\frac{p\left(j_l+i\right)+1}{2}} \;\;\;,\;\;\; i = 1,2,\ldots,q_l,
\end{equation*}
are the knots in segment $S_l$ coalescing with data abscissae. So the segments $S_0$ (to the left of first interior knot $t_{j_1}$) and $S_r$ (to the right of the last interior knot $t_{j_r}$) are left to be defined here (and also for $r=1$). $S_0$ consists of abscissae $x_0,\ldots,x_{i_1}$ and contains - if $j_1 - 1 > 0$ - the knots 
\begin{equation*}
t_i = x_{\frac{p(i)+1}{2}} \;\;\;,\;\;\; i = 1,2,\ldots,j_1-1 \quad. 
\end{equation*}
Segment $S_r$ consists of abscissae $x_{i_r+1},\ldots,x_{\mu+1}$ and contains - if $j_r+1 \leq k$ - the knots
\begin{equation*}
t_i = x_{\frac{p(i)+1}{2}} \;\;\;,\;\;\; i = j_r+1,\ldots,k \quad. 
\end{equation*}
\end{itemize}
\end{defineCorrollary}

\begin{figure}[h]
\begin{center}
\resizebox{1.0\textwidth}{!}{
\begin{tikzpicture}

\draw[fill,lightgray] (3.75,-3)--(3.75,2.3)--(12,2.3)--(12,-3)--(3.75,-3);
\draw[darkgray] (7.75,-2.7) node {\large S e g m e n t \; $S_l$};
\draw[<-,>=stealth,darkgray] (3.75,-2.7)--(6,-2.7);
\draw[->,>=stealth,darkgray] (9.5,-2.7)--(12,-2.7);

\draw (-0.3,0)--(1.7,0);

\draw (0,0.15)--(0,-0.15);
\draw (0,-0.5) node {$x_0$};

\draw[<-,>=stealth] (0.375,0.4)--(0.375,0.8) node[above=2pt] {\rotatebox{45}{$0$}};
\draw[<-,>=stealth] (0.75,0.4)--(0.75,0.8) node[above=2pt] {\rotatebox{45}{$1$}};

\draw (0.75,0.15)--(0.75,-0.15);
\draw (0.75,-0.5) node {$x_1$};

\draw[<-,>=stealth] (1.125,0.4)--(1.125,0.8) node[above=2pt] {\rotatebox{45}{$2$}};
\draw[<-,>=stealth] (1.5,0.4)--(1.5,0.8) node[above=2pt] {\rotatebox{45}{$3$}};

\draw (1.5,0.15)--(1.5,-0.15);
\draw (1.5,-0.5) node {$x_2$};

\draw[dashed] (1.7,0)--(2.8,0);
\draw (2.8,0)--(4.7,0);

\draw (3,0.15)--(3,-0.15);
\draw (3,-0.5) node {$x_{i_l}$};

\draw[<-,>=stealth] (3,0.4)--(3,0.8);
\draw (3.2,0.8) node[above=2pt] {\rotatebox{45}{$2i_l-1$}};

\draw[<-,>=stealth] (3.75,0.4)--(3.75,0.8) node[above=2pt] {\rotatebox{45}{$2i_l$}};
\draw (3.75,0) circle (0.15cm);
\draw[<-,>=stealth] (3.75,-1)--(3.75,-1.4) node[below=2pt] {$t_{j_l}$};

\draw[<-,>=stealth] (4.5,0.4)--(4.5,0.8);
\draw (4.7,0.8) node[above=2pt] {\rotatebox{45}{$2i_l+1$}};
\draw (4.5,0.15)--(4.5,-0.15);
\draw (4.5,-0.5) node {$x_{i_l+1}$};

\draw[dashed] (4.7,0)--(6.0,0);
\draw (6.0,0)--(7.3,0);

\draw[<-,>=stealth] (6.5,0.4)--(6.5,0.8);
\draw (6.9,0.8) node[above=2pt] {\rotatebox{45}{$p\left(j_l+1\right)$}};
\draw (6.5,0.15)--(6.5,-0.15);
\draw (6.5,0) circle (0.15cm);
\draw (6.5,-0.5) node {$x_{\frac{p\left(j_l+1\right)+1}{2}}$};
\draw[<-,>=stealth] (6.5,-1)--(6.5,-1.4) node[below=2pt] {$t_{j_l+1}$};

\draw[dashed] (7.0,0)--(8.3,0);
\draw (8.3,0)--(9.3,0);

\draw[<-,>=stealth] (8.8,0.4)--(8.8,0.8);
\draw (9.2,0.8) node[above=2pt] {\rotatebox{45}{$p\left(j_l+q_l\right)$}};
\draw (8.8,0.15)--(8.8,-0.15);
\draw (8.8,0) circle (0.15cm);
\draw (8.8,-0.5) node {$x_{\frac{p\left(j_l+q_l\right)+1}{2}}$};
\draw[<-,>=stealth] (8.8,-1)--(8.8,-1.4) node[below=2pt] {$t_{j_l+q_l}$};

\draw[dashed] (9.3,0)--(10.7,0);
\draw (10.6,0)--(13.3,0);

\draw[<-,>=stealth] (11,0.4)--(11,0.8);
\draw (11.4,0.8) node[above=2pt] {\rotatebox{45}{$2i_{l+1}-1$}};
\draw (11,0.15)--(11,-0.15);
\draw (11,-0.5) node {$x_{i_{l+1}}$};

\draw[<-,>=stealth] (12,0.4)--(12,0.8);
\draw (12.2,0.8) node[above=2pt] {\rotatebox{45}{$2i_{l+1}$}};
\draw (12,0) circle (0.15cm);
\draw[<-,>=stealth] (12,-1)--(12,-1.4) node[below=2pt] {$t_{j_{l+1}}$};

\draw[<-,>=stealth] (13,0.4)--(13,0.8);
\draw (13.4,0.8) node[above=2pt] {\rotatebox{45}{$2i_{l+1}+1$}};
\draw (13,0.15)--(13,-0.15);
\draw (13,-0.5) node {$x_{i_{l+1}+1}$};

\draw[dashed] (13.3,0)--(14.4,0);
\draw (14.4,0)--(17.3,0);

\draw[<-,>=stealth] (14.6,0.4)--(14.6,0.8);
\draw (14.8,0.8) node[above=2pt] {\rotatebox{45}{$2\mu-3$}};
\draw (14.6,0.15)--(14.6,-0.15);
\draw (14.6,-0.5) node {$x_{\mu-1}$};

\draw[<-,>=stealth] (15.3,0.4)--(15.3,0.8);
\draw (15.5,0.8) node[above=2pt] {\rotatebox{45}{$2\mu-2$}};
\draw[<-,>=stealth] (16,0.4)--(16,0.8);
\draw (16.2,0.8) node[above=2pt] {\rotatebox{45}{$2\mu-1$}};
\draw (16,0.15)--(16,-0.15);
\draw (16,-0.5) node {$x_{\mu}$};
\draw[<-,>=stealth] (16.5,0.4)--(16.5,0.8);
\draw (16.6,0.8) node[above=2pt] {\rotatebox{45}{$2\mu$}};

\draw (17,0.15)--(17,-0.15);
\draw (17,-0.5) node {$x_{\mu+1}$};

\end{tikzpicture}
} 
\end{center}
\caption{Diagram of segment $S_l$: Numbers $0,1,2,\ldots,2\mu$ as pointer for the position of knots and the segment $S_l = \big\{x_{i_l+1},\ldots,x_{i_{l+1}}\big\}$ ranging from interior knot $t_{j_l}$ to interior knot $t_{j_{l+1}}$.}
\label{fig2}
\end{figure}
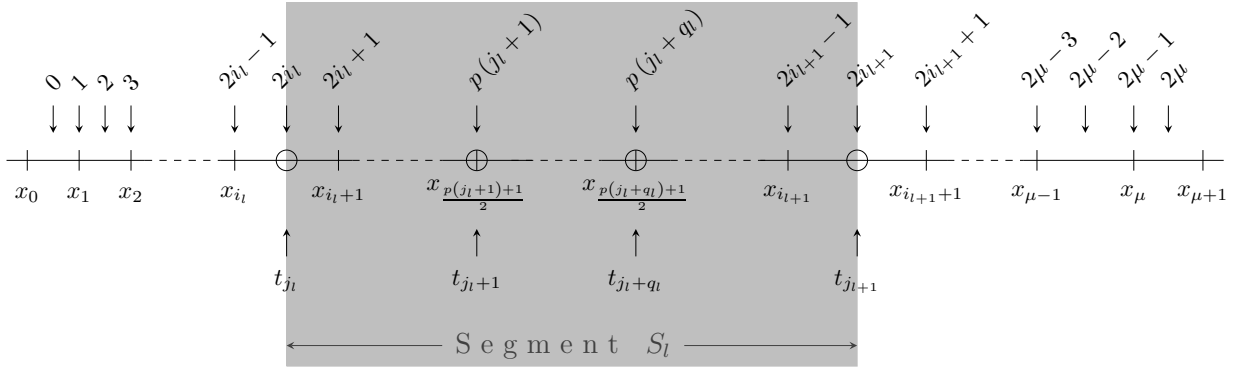


\vspace{0.1cm}
\begin{exmp}
Position vector $p = \left(6,1,3,19,24,29\right)$ for spline $s^{\ast}$ in Fig. \ref{fig1} leads to segments \mbox{$S_0 = \left\{x_0,x_1,x_2,x_3\right\}$} without knots, $S_1 = \left\{x_4,x_5,\ldots,x_{12}\right\}$ with knots $t_2 := x_7$ and $t_3 := x_{10}$ and to segment $S_2 = \left\{x_{13},x_{14},\ldots,x_{18}\right\}$ with the knot $t_5 := x_{15}$. 
\end{exmp}

\vspace{0.1cm}
\begin{remark}\label{eq:bemerkungMindesanforderungen}
Because of Theorem 12 in \cite{cromme03} we can limit our search for a best approximation to splines with only simple knots and for which the additional features (a) to (d) from Theorem 12 in \cite{cromme03} hold; see also Appendix A. In Lemma \ref{eq:LemmaAequivalenz} we will prove that these are exactly the splines with regular position vectors.
\end{remark}

\begin{definition}\label{eq:defRegularPositionsvektor}
A position vector $p \in \mathbb{N}_0^k$ is called \textit{regular position vector of length $k$}, if the following features hold:
\begin{itemize}
  \item[(a)] $\left\{p(1),\ldots,p(k)\right\} \subseteq \left\{1,2,\ldots,2\mu-1\right\}$ \quad and
	
	\item[(b)] $\text{$p(j)$ even} \quad\Rightarrow\quad \forall i\in\{1,\ldots,j-1\}: \; p(i) \neq p(j) - 1,\; \forall i\in\{j+1,\ldots,k\}: \; p(i) \neq p(j)+1$\;.
	
  \item[(c)] For $p(j)$ and $p(j+l+1)$ ($l\geq 0$) both even and for $l \geq 1$ in addition $p(j+1),\ldots,p(j+l)$ odd the following inequality holds 
\begin{equation*}
\frac{p(j+l+1)}{2} - \frac{p(j)}{2} \;\geq\; 2 + l \quad.
\end{equation*}
	
	\item[(d)] $\forall\, j\in\{1,\ldots,k-1\}: \;\; p(j+1)-p(j) \geq \left\{
\begin{array}{lcl}
4 &,& \text{for $p(j)$, $p(j+1)$ both even} \\
2 &,& \text{otherwise}
\end{array}
\right.$	
\end{itemize}
\end{definition}

\vspace{0.2cm}
\begin{lemma}\label{eq:LemmaAequivalenz}
The position vector of a spline $s\in S^k[a,b]$ is regular iff $s$ exhibits features (a) thru (d) from Theorem 12 in \cite{cromme03}.
\end{lemma}
\begin{proof}
Lemma \ref{eq:LemmaAequivalenz} reflects the fact that properties
(a) to (d) in Definition \ref{eq:defRegularPositionsvektor} for components $p(1),\ldots,p(k)$ guarantee features (a) to (d) in Theorem 12 in \cite{cromme03} and vice versa. As an example let us elaborate (d):

\vspace{0.1cm}
\noindent{}So let $p$ be a position vector of a spline $s\in S_k^1[a,b]$, $t_j$ and $t_{j+1}$ knots of $s$. 
\begin{itemize}
	\item[-] Case 1: $t_j$ and $t_{j+1}$are interior knots, i.e. $p(j)$ and $p(j+1)$ are even. Then at least two data abscissae are situated between $t_j$ and $t_{j+1}$ iff
\begin{equation*}
\frac{p(j+1)}{2} - \left(\frac{p(j)}{2}+1\right) + 1 \;\geq\; 2 \quad\quad\Leftrightarrow\quad\quad p(j+1) - p(j) \;\geq\; 4 \quad.
\end{equation*}
  \item[-] Case 2: $t_j$ is an interior and $t_{j+1}$ a data knot, i.e., $p(j)$ is even and $p(j+1)$ is odd. Then at least two data abscissae lie between or on $t_j$ and $t_{j+1}$ iff 
\begin{equation*}
\frac{p(j)}{2} + 1 \;<\; \frac{p(j+1)+1}{2} \Leftrightarrow p(j+1) - p(j) \;>\; 1 \Leftrightarrow p(j+1) - p(j) \;\geq\; 2 \;\;.
\end{equation*}
\end{itemize}
The equivalence of feature (d) from Theorem 12 in \cite{cromme03} to the requirement
\begin{equation*}
p(j+1) - p(j) \;\geq\; 2
\end{equation*}
is shown with similar arguments in
\begin{itemize}
	\item[-] Case 3: $t_j$ is data knot, $t_{j+1}$ interior knot, that is, $p(j)$ is odd and $p(j+1)$ is even,
	\item[-] Case 4: $t_j$ and $t_{j+1}$ are data knots, that is, $p(j)$ and $p(j+1)$ are odd.
\end{itemize}
Altogether this proves the equivalence of \ref{eq:defRegularPositionsvektor} (d) and Theorem 12 (d) from \cite{cromme03}. \hfill $\square$
\end{proof}

\vspace{0.2cm}
\begin{corollary}
For a regular position vector the following statements hold true:
\begin{equation*}
\big\{p(1),\ldots,p(k)\big\}\subseteq\big\{1,2,\ldots,2\mu-1\big\} \quad\text{and}\quad p(1) < p(2) < \ldots < p(k) \quad.
\end{equation*}
\end{corollary}
\begin{proof}
See Definition \ref{eq:defRegularPositionsvektor} (a) and (d).
\end{proof}

\vspace{0.2cm}
\begin{lemma}\label{eq:posvekanzahl}
The number of regular position vectors grows exponentially in $\mu$.
\end{lemma}
\begin{proof}
According to Definition \ref{eq:defRegularPositionsvektor} a position vector is regular if all $k$ components $p(1) < p(2) < \ldots < p(k)$ are odd, i.e. 
\begin{equation*}
p(j) \in \big\{1,3,5,\ldots,2\mu-3,2\mu-1\big\}\;,\;\; j=1,\ldots,k \quad.
\end{equation*}
In such cases all knots of related splines are data knots. So there are at least as many regular position vectors as there are subsets of  $\left\{1,3,5,\ldots,2\mu-1\right\}$ with $k$ elements. According to standard results in combinatorics there are $\binom{\mu}{\mu-k} = \frac{\mu!}{(\mu-k)!k!}$ such subsets. With growing $\mu$ and $k=\frac{\mu}{2}$ for example exponential growth of the number of regular position vectors is therefore programmed. \hfill $\square$
\end{proof}

\vspace{0.2cm}
The fast growth of the number of regular position vectors is evident in Table \ref{tab:tabRegPosVek}:

\begin{table}
\caption{Number of regular position vectors as a function of knots and data points.}
\label{tab:tabRegPosVek}
\begin{center}
\begin{tabular}{r|rrrrrr}  
\hline\noalign{\smallskip}
number of & \multicolumn{6}{c}{number of data points $\mu+2$} \\                                     
knots $k$ & 15 & 20 & 25 & 30 & 35 & 40 \\
\noalign{\smallskip}\hline\noalign{\smallskip}                    
1 & 25 & 35 & 45 & 55 & 65 & 75 \\                            
2 & 265 & 545 & 925 & 1405 & 1985 & 2665 \\                   
3 & 1561 & 4991 & 11521 & 22151 & 37881 & 59711 \\            
4 & 5641 & 29961 & 97281 & 241601 & 506921 & 947241 \\        
5 & 13073 & 124515 & 590557 & 1937199 & 5060441 & 11326283 \\ 
6 & 19825 & 369305 & 2668525 & 11847485 & 39146185 & 106114625 \\  
7 & 19825 & 795455 & 9173505 & 56610575 & 249673265 & 799538175 \\
\hline\noalign{\smallskip}                 
\end{tabular}
\end{center}
\end{table}

\vspace{0.2cm}
We can now specify the algorithm for the solution of Approximation problem \ref{eq:approxprobelmdef}: The algorithm examines systematically all regular position vectors to see if they might code the position of the knots of a best approximating spline $s\in S^1_k[a,b]$.
\begin{algorithm}\label{eq:basisalgorithmus}
\mbox{}
\begin{itemize}
	\item[1.] Set $\varrho := \infty$.
	\item[2.] For the next regular position vector $p$ of length $k$ for the data abscissae $x_0,\ldots,x_{\mu+1}$ which has not yet been examined carry out exactly one of the following two steps 3 and 4:
	
	\item[3.] If all $p(i)$ for $i=1,2,\ldots,k$ are odd compute the best approximation \newline $s \in S^1\left(x_{\frac{p(1)-1}{2}},\ldots,x_{\frac{p(k)-1}{2}}\right)$ to the data $\left(x_0,f_0\right),\ldots,\left(x_{\mu+1},f_{\mu+1}\right)$. The spline $s$ is thus the solution of a discrete linear least-squares approximation problem with fixed knots. If $\big\|F-s(X)\big\|_2 < \varrho$, set $\varrho := \big\|F-s(X)\big\|_2$, $s^\ast := s$. Continue with step 2.
	
	\item[4.] If at least one component $p(i)$ is even, i.e., for a suitable $r\in\{1,\ldots,k\}$ there are exactly $r$ interior knots, carry out the following steps; notations are as in Definition \ref{definitonSegmente}, in addition $i_0 := -1, i_{r+1}:=\mu+2, j_0 := 0, j_{r+1} := k+1$, \mbox{$q_0 := j_1 -1$}, $q_r := k-j_r$:
\begin{itemize}
	\item[(a)] For each segment $S_l$ $\left(l=0,1,\ldots,r\right)$ compute the best least-squares approximation $s_l$ to the data $\left(x_{i_l+1},f_{i_l+1}\right), \ldots, \left(x_{i_{l+1}},f_{i_{l+1}}\right)$. \; Here $s_l$ is to chosen from $\mathbb{P}_1$ if $q_l = 0$ and
\begin{equation*}
s_l\in S^1\Big(t_{j_l+1},\ldots,t_{j_l+q_l}\Big) \;\;=\;\; S^1\left(x_{\frac{p\left(j_l+1\right)+1}{2}},\ldots,x_{\frac{p\left(j_{l+1}-1\right)+1}{2}}\right)
\end{equation*}
if $q_l > 0$.
	\item[(b)] For all $l=1,\ldots,r$ check whether the splines $s_{l-1}$ and $s_l$ intersect in some point $z_l \in \left(x_{i_l},x_{i_l+1}\right)$. If this is the case for all $l$, $l=1,\ldots,r$ define
\begin{equation*}
s(x) := \left\{
\begin{array}{lcl}
s_0(x) & , & x_0 \leq x \leq z_1 \\
s_l(x) & , & z_l < x \leq z_{l+1} \;,\; 1 \leq l \leq r-1 \\
s_r(x) & , & z_r < x \leq x_{\mu+1} \hspace{3.0cm}.
\end{array}
\right.
\end{equation*}
If $\big\|F-s(X)\big\|_2 < \varrho$ set $\varrho := \big\|F-s(X)\big\|_2$, $s^\ast := s$ and procede with step 2. If for some $l\in\{1,\ldots,r\}$ splines $s_{l-1}$ and $s_l$ do not intersect in $\left(x_{i_l},x_{i_l+1}\right)$ reject the position vector $p$ under investigation and continue with step 2.
\end{itemize}
Procede with step 2.

  \item[5.] When all regular position vectors $p$ of length $k$ have been examined $s^\ast$ is the wanted best approximation and $\varrho$ the corresponding approximation error. End.
\end{itemize}
\end{algorithm}


\vspace{0.2in}
\noindent{}The proof that Algorithm \ref{eq:basisalgorithmus} solves the approximation problem is based on Theorem 12 in \cite{cromme03}:
\begin{theorem}\label{eq:konvergenzsatz}
Algorithm \ref{eq:basisalgorithmus} stops after a finite number of steps and delivers with $s^{\ast}$ a best approximation 
\begin{equation*}
\big\|f-s^\ast\big\|_{2,X} \;\;=\;\; \inf_{s\in S^1_k[x_0,x_{\mu+1}]}\big\|f-s\big\|_{2,X} \quad.
\end{equation*}
\end{theorem}
\begin{proof}
Let $s^{\ast}$ be a best approximation whith the additional features from Theorem 12 in \cite{cromme03}. Because of Lemma \ref{eq:LemmaAequivalenz} the position vector $p^{\ast}$ of $s^{\ast}$ is then regular and is therefore treated in Algorithm \ref{eq:basisalgorithmus}.
\end{proof}

\vspace{0.1cm}
\noindent{}Furthermore, the following assertion holds:
\begin{assertion}
When Algorithm \ref{eq:basisalgorithmus} elaborates on $p^{\ast}$ the function computed in step 3 or step 4 by the algorithm ist just $s^{\ast}$.
\end{assertion}

\vspace{0.1cm}
\begin{proof}
a) If step 3 is executed (all $p^\ast(i)$ are odd) then $s^{\ast}$ as best approximation from $S^1_k\left[a,b\right]$ is even more best approximation from the smaller set $S^1\left(x_{\frac{p(1)+1}{2}},\ldots,x_{\frac{p(k)+1}{2}}\right)$ which $s^{\ast}$ is also an element of.

b) If step 4 is executed every approximation $s_i$ on a segment is identical with $s^{\ast}$ in its domain. For simplicity we demonstrate this only for the first segment $S_0$. 

Assumption: $s_0$ is a better linear approximation to the data $\left(x_0,f_0\right),\ldots$, $\left(x_{\frac{p\left(j_1\right)}{2}}, f_{\frac{p\left(j_1\right)}{2}}\right)$ than $s^{\ast}$. Any sufficiently small shift of $s^{\ast}$ in the direction of $s_0$ improves the approximation error.  Because $\left(s^{\ast}\right)'$ has a jump discontinuity in $t_{j_1}$ (see Theorem 12 (h) in \cite{cromme03}), the modification of $s^\ast$ in direction $s_0$ as described above results in a $s^{\ast}_{new}\in S_k^1[a,b]$ with $\left\|F-s^{\ast}_{neu}(X)\right\|_2 < \left\|F-s^{\ast}(X)\right\|_2$ contradicting the optimality of $s^{\ast}$.

The above assumption must therefore be wrong and the assertion b) holds. $s^\ast$ is thus an element of $K \subset S^1_k[x_0,x_{\mu+1}]$, the splines computed by the algorithm. Therefore the minimal approximation error is realized by the algorithm and a best approximation computed. \hfill $\square$
\end{proof}

\vspace{0.2cm}

\setcounter{equation}{0}
\section{Realization, simplification and acceleration}
\label{sec:pc}
As can be seen from Lemma \ref{eq:posvekanzahl} the computational cost and computing time grows exponentially with the number of data $\mu+2$ causing a corresponding memory demand. It ist advisable to restrain from computing and storing all (regular) position vectors since this would cause an exponential demand for memory. Instead, when a newly regular position vector is found, steps 3 und 4 should follow right away such that the position vector - once examined - can be deleted.

The algorithm is suitable for parallelization. Steps 3 and 4 can be executed for several regular position vectors in parallel.

Segments $S_0$ to $S_r$ and their best approximations $s_0,\ldots,s_r$ could also be examined in parallel for intersection points. But this implementation is not advisable because it might lead to unnecessary computations: If the approxiomations $s_0,s_1,\ldots,s_r$ on segments $S_0,S_1,\ldots,S_r$ are computed in parallel and for example $s_0$ and $s_1$ do not intersect in $\big(x_{i_0},x_{i_0+1}\big)$, then the position vector under examination can already be dismissed and the computation of $s_2,\ldots,s_r$ (for $r \geq 2$) in the parallel realization of the algorithm would be an additional and unnecessary effort.

In this respect more effective is sequential programming, starting step 4, computing $s_0$ and $s_1$, and then check whether they intersect in  $\big(x_{i_0},x_{i_0+1}\big)$. In analogy a best approximation $s_{l+2}$ on the next segment $S_{l+2}$ is determined only if $s_l$ and $s_{l+1}$ have been found to intersect in $\big(x_{i_l},x_{i_l+1}\big)$. Otherwise, the position vector currently under investigation is sorted out and computing further approximations related to this position vector can be skipped.

Furthermore, the algorithm should be programmed such that if the program terminates prematurely because of time limitations at least the best approximation found so far is saved and handed over to the main program. If position vectors of good approximations are dealt with early the consequences of a premature termination can be further limited.

If there is prior knowledge about the position of knots of a best approximation related position vectors should be examined with priority.

Which clues for the optimal position of knots can be deduced from the data? Besides the divided differences of first order
\begin{equation*}
\Delta_j^{(1)} \;:=\; \frac{f_{j+1}-f_j}{x_{j+1}-x_j} \;\;\;,\;\; 0\leq j\leq \mu
\end{equation*}
the (central) second order divided differences
\begin{equation*}
\Delta_j^{(2)} \;:=\; \frac{f_{j+1} \,-\,2 f_j \,+\,f_{j-1}}{\left(x_{j+1} \,-\, x_j\right)^2} \;\;\;,\;\; 1 \leq j \leq \mu \quad.
\end{equation*}
can be helpful here. The knots of a best-approximating first-order spline could be expected in the neighborhood of data abscissae $x_j$ where the $\Delta_j^{(1)}$ vary heavily and the absolute values of the $\Delta_j^{(2)}$ are "large". Related positions vectors should therefore be examined with priority.

\vspace{0.2cm}
\setcounter{equation}{0}
\label{sec:examples}
\section{Numerical examples and applications}
In this section we present numerical examples to illustrate the performance of Algorithm \ref{eq:basisalgorithmus}. For each example we list the knots and the approximation error and present the solutions in graphical form, but refrain from enumerating the coefficients of the best approximations in all examples.

\begin{exmp}\label{allgBsp1}
Let $\mu=15$, $x_i := i$ for $i=0,1,\ldots,16$, $f_i := 1$ for $i=0,1,\ldots,7,9,\ldots,16$ and for $f_8 := 2$ be given. For $k = 1,2,3,4,5$ Algorithm \ref{eq:basisalgorithmus} computes best approximations $s_k \in S^1_k\left[t_0,t_{16}\right]$ with knots $t_1,\ldots,t_k$ given in the following table:
\begin{equation*}
\begin{array}{l|lllll|r}
 & t_1 & t_2 & t_3 & t_4 & t_5 & \left\|F-s_k(X)\right\|_2 \\
\hline
\hline
k=1 & 8 & & & & & 0.87586 \\
k=2 & 7 & 8 & & & & 0.78881 \\
k=3 & 7 & 8 & 9 & & & 0 \\
k=4 & 4 & 7 & 8 & 9 & & 0 \\
k=5 & 1 & 2 & 7 & 8 & 9 & 0
\end{array}
\end{equation*}
For $k=1,2,3$ $s_k$ has exactly $k$ proper knots. To approximate the data for $k \geq 4$ only $3$ proper knots are needed and the remaining $k-3$ knots are not proper knots; see Fig. \ref{fig3}.
\end{exmp}

\begin{figure}[h]
\begin{center}
\scalebox{0.75}{
\begin{picture}(445,180)
\color{black}
\put(15,20){\vector(1,0){355}}
\put(20,15){\vector(0,1){163}}
\put(17,4){\small $0$}
\put(40.9375,16){\line(0,1){8}}
\put(30.4688,18){\line(0,1){4}}
\put(37.9375,4){\small $1$}
\put(61.875,16){\line(0,1){8}}
\put(51.4063,18){\line(0,1){4}}
\put(58.875,4){\small $2$}
\put(82.8125,16){\line(0,1){8}}
\put(72.3438,18){\line(0,1){4}}
\put(79.8125,4){\small $3$}
\put(103.75,16){\line(0,1){8}}
\put(93.2813,18){\line(0,1){4}}
\put(100.75,4){\small $4$}
\put(124.6875,16){\line(0,1){8}}
\put(114.2188,18){\line(0,1){4}}
\put(121.6875,4){\small $5$}
\put(145.625,16){\line(0,1){8}}
\put(135.1563,18){\line(0,1){4}}
\put(142.625,4){\small $6$}
\put(166.5625,16){\line(0,1){8}}
\put(156.0938,18){\line(0,1){4}}
\put(163.5625,4){\small $7$}
\put(187.5,16){\line(0,1){8}}
\put(177.0313,18){\line(0,1){4}}
\put(184.5,4){\small $8$}
\put(208.4375,16){\line(0,1){8}}
\put(197.9688,18){\line(0,1){4}}
\put(205.4375,4){\small $9$}
\put(229.375,16){\line(0,1){8}}
\put(218.9063,18){\line(0,1){4}}
\put(223.375,4){\small $10$}
\put(250.3125,16){\line(0,1){8}}
\put(239.8438,18){\line(0,1){4}}
\put(244.3125,4){\small $11$}
\put(271.25,16){\line(0,1){8}}
\put(260.7813,18){\line(0,1){4}}
\put(265.25,4){\small $12$}
\put(292.1875,16){\line(0,1){8}}
\put(281.7188,18){\line(0,1){4}}
\put(286.1875,4){\small $13$}
\put(313.125,16){\line(0,1){8}}
\put(302.6563,18){\line(0,1){4}}
\put(307.125,4){\small $14$}
\put(334.0625,16){\line(0,1){8}}
\put(323.5938,18){\line(0,1){4}}
\put(328.0625,4){\small $15$}
\put(355,16){\line(0,1){8}}
\put(344.5313,18){\line(0,1){4}}
\put(349,4){\small $16$}
\put(365,4){\small $x$}
\put(16,68.75){\line(1,0){8}}
\put(7,65.75){\small $1$}
\put(16,166.25){\line(1,0){8}}
\put(18,117.5){\line(1,0){4}}
\put(7,163.25){\small $2$}
\color{black}
\put(20,68.75){\circle*{4}}
\put(40.9375,68.75){\circle*{4}}
\put(61.875,68.75){\circle*{4}}
\put(82.8125,68.75){\circle*{4}}
\put(103.75,68.75){\circle*{4}}
\put(124.6875,68.75){\circle*{4}}
\put(145.625,68.75){\circle*{4}}
\put(166.5625,68.75){\circle*{4}}
\put(187.5,166.25){\circle*{4}}
\put(208.4375,68.75){\circle*{4}}
\put(229.375,68.75){\circle*{4}}
\put(250.3125,68.75){\circle*{4}}
\put(271.25,68.75){\circle*{4}}
\put(292.1875,68.75){\circle*{4}}
\put(313.125,68.75){\circle*{4}}
\put(334.0625,68.75){\circle*{4}}
\put(355,68.75){\circle*{4}}                                
\dottedline[\circle*{1}]{5}(20,59.4007)(187.5,91.4555)(355,59.4007)
\thicklines 
\drawline[-35](20,68.75)(166.5625,68.75)
\drawline[-30](166.5625,68.75)(187.5,105.5833)
\drawline[-35](187.5,105.5833)(355,53.5833)
\thinlines
\dottedline[\circle*{1}](20,68.75)(166.5625,68.75)(187.5,166.25)(208.4375,68.75)(355,68.75)
\color{black}
\put(396,170){\circle*{4}}
\put(412,167){$\left(t_i,f_i\right)$}
\dottedline[\circle*{1}]{5}(384,155)(410,155)
\thicklines
\drawline[-30](384,140)(410,140)
\thinlines
\dottedline[\circle*{1}](384,125)(410,125)
\put(415,152){\small $k=1$}
\put(415,137){\small $k=2$}
\put(415,122){\small $k \geq 3$}
\end{picture}
} 
\end{center}
\caption{Best approximation from Example \ref{allgBsp1}.}
\label{fig3}
\end{figure}
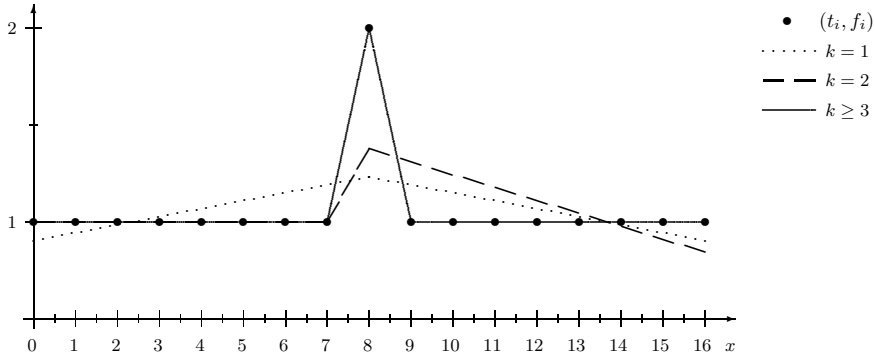

\vspace{0.2in}
\begin{exmp}\label{allgBsp2}
The so-called Titanium Heat Data have been dicused in literature and approximated from a number of approximation families; e.g. \cite{deBoor} and \cite{jupp}. The data describe a property of titanium as function of temperature. The measuring points are equidistant in the interval $[595,1075]$. The least-squares approximation with broken lines delivers a good approximation as can seen in Fig. \ref{fig4} and the following table::
\begin{small}
\begin{equation*}
\begin{array}{r||r||c|rrrrr}
k & \|F-s_k(X)\|_2 & i & 1 & 2 & 3 & 4 & 5 \\
\hline
\hline
3 & 0.2632 & t_i & 858.4883 & 897.8327 & 940.2917 &  &  \\ 
 &  & s_k(t_i) & 0.7642 & 2.3065 & 0.6659 &  &  \\ 
\hline
4 & 0.1875 & t_i & 831.4392 & 866.8552 & 897.5429 & 940.2917 &  \\ 
 &  & s_k(t_i) & 0.7074 & 1.0396 & 2.3177 & 0.6659 &  \\ 
\hline
5 & 0.1349 & t_i & 831.4392 & 866.8552 & 898.3019 & 930.6129 & 958.3397 \\ 
 &  & s_k(t_i) & 0.7074 & 1.0396 & 2.3494 & 0.9535 & 0.6153 \\ 
\end{array}
\end{equation*}
\end{small}
\end{exmp}

\begin{figure}[h]
\begin{center}
\resizebox{1.0\textwidth}{!}{
\begin{tikzpicture}     
\draw[->,color=black,>=stealth] (-0.2,0) -- (15.5,0) node[below=2pt] {\footnotesize $x$}; 
\draw[->,color=black,>=stealth] (0,-0.2) -- (0,7.2);     
\draw (0,0-.2)--(0,0) node[below=6pt]{\small 595};    
\draw (0.625,0-.1)--(0.625,0);      
\draw (1.25,0-.2)--(1.25,0) node[below=6pt]{\small 635};    
\draw (1.875,0-.1)--(1.875,0);      
\draw (2.5,0-.2)--(2.5,0) node[below=6pt]{\small 675};   
\draw (3.125,0-.1)--(3.125,0);      
\draw (3.75,0-.2)--(3.75,0) node[below=6pt]{\small 715};    
\draw (4.375,0-.1)--(4.375,0);      
\draw (5,0-.2)--(5,0) node[below=6pt]{\small 755};    
\draw (5.625,0-.1)--(5.625,0);      
\draw (6.25,0-.2)--(6.25,0) node[below=6pt]{\small 795};    
\draw (6.875,0-.1)--(6.875,0);      
\draw (7.5,0-.2)--(7.5,0) node[below=6pt]{\small 835};   
\draw (8.125,0-.1)--(8.125,0);      
\draw (8.75,0-.2)--(8.75,0) node[below=6pt]{\small 875};    
\draw (9.375,0-.1)--(9.375,0);      
\draw (10,0-.2)--(10,0) node[below=6pt]{\small 915};     
\draw (10.625,0-.1)--(10.625,0);    
\draw (11.25,0-.2)--(11.25,0) node[below=6pt]{\small 955};  
\draw (11.875,0-.1)--(11.875,0);    
\draw (12.5,0-.2)--(12.5,0) node[below=6pt]{\small 995};    
\draw (13.125,0-.1)--(13.125,0);    
\draw (13.75,0-.2)--(13.75,0) node[below=6pt]{\small 1035};    
\draw (14.375,0-.1)--(14.375,0);    
\draw (15,0-.2)--(15,0) node[below=6pt]{\small 1075};    
\draw (0-.2,0)--(0,0) node[left=6pt]{\small 0.5};     
\draw (0-.1,0.18421)--(0,0.18421);     
\draw (0-.2,0.36842)--(0,0.36842) node[left=6pt]{\small 0.6};  
\draw (0-.1,0.55263)--(0,0.55263);     
\draw (0-.2,0.73684)--(0,0.73684) node[left=6pt]{\small 0.7};  
\draw (0-.1,0.92105)--(0,0.92105);     
\draw (0-.2,1.1053)--(0,1.1053) node[left=6pt]{\small 0.8};    
\draw (0-.1,1.2895)--(0,1.2895);    
\draw (0-.2,1.4737)--(0,1.4737) node[left=6pt]{\small 0.9};    
\draw (0-.1,1.6579)--(0,1.6579);    
\draw (0-.2,1.8421)--(0,1.8421) node[left=6pt]{\small 1};   
\draw (0-.1,2.0263)--(0,2.0263);    
\draw (0-.2,2.2105)--(0,2.2105) node[left=6pt]{\small 1.1};    
\draw (0-.1,2.3947)--(0,2.3947);    
\draw (0-.2,2.5789)--(0,2.5789) node[left=6pt]{\small 1.2};    
\draw (0-.1,2.7632)--(0,2.7632);    
\draw (0-.2,2.9474)--(0,2.9474) node[left=6pt]{\small 1.3};    
\draw (0-.1,3.1316)--(0,3.1316);    
\draw (0-.2,3.3158)--(0,3.3158) node[left=6pt]{\small 1.4};    
\draw (0-.1,3.5)--(0,3.5);    
\draw (0-.2,3.6842)--(0,3.6842) node[left=6pt]{\small 1.5};    
\draw (0-.1,3.8684)--(0,3.8684);    
\draw (0-.2,4.0526)--(0,4.0526) node[left=6pt]{\small 1.6};    
\draw (0-.1,4.2368)--(0,4.2368);    
\draw (0-.2,4.4211)--(0,4.4211) node[left=6pt]{\small 1.7};    
\draw (0-.1,4.6053)--(0,4.6053);    
\draw (0-.2,4.7895)--(0,4.7895) node[left=6pt]{\small 1.8};    
\draw (0-.1,4.9737)--(0,4.9737);    
\draw (0-.2,5.1579)--(0,5.1579) node[left=6pt]{\small 1.9};    
\draw (0-.1,5.3421)--(0,5.3421);    
\draw (0-.2,5.5263)--(0,5.5263) node[left=6pt]{\small 2};   
\draw (0-.1,5.7105)--(0,5.7105);    
\draw (0-.2,5.8947)--(0,5.8947) node[left=6pt]{\small 2.1};    
\draw (0-.1,6.0789)--(0,6.0789);    
\draw (0-.2,6.2632)--(0,6.2632) node[left=6pt]{\small 2.2};    
\draw (0-.1,6.4474)--(0,6.4474);    
\draw (0-.2,6.6316)--(0,6.6316) node[left=6pt]{\small 2.3};    
\draw (0-.1,6.8158)--(0,6.8158);    
\draw (0-.2,7)--(0,7) node[left=6pt]{\small 2.4};     
\color{black}     
\draw[fill,black] (0,0.53053) circle (1.5pt);   
\draw[fill,black] (0.3125,0.44947) circle (1.5pt);    
\draw[fill,black] (0.625,0.50842) circle (1.5pt);     
\draw[fill,black] (0.9375,0.54895) circle (1.5pt);    
\draw[fill,black] (1.25,0.56) circle (1.5pt);   
\draw[fill,black] (1.5625,0.51211) circle (1.5pt);    
\draw[fill,black] (1.875,0.53789) circle (1.5pt);     
\draw[fill,black] (2.1875,0.57842) circle (1.5pt);    
\draw[fill,black] (2.5,0.56) circle (1.5pt);    
\draw[fill,black] (2.8125,0.57105) circle (1.5pt);    
\draw[fill,black] (3.125,0.53053) circle (1.5pt);     
\draw[fill,black] (3.4375,0.60053) circle (1.5pt);    
\draw[fill,black] (3.75,0.60053) circle (1.5pt);   
\draw[fill,black] (4.0625,0.61895) circle (1.5pt);    
\draw[fill,black] (4.375,0.64842) circle (1.5pt);     
\draw[fill,black] (4.6875,0.64842) circle (1.5pt);    
\draw[fill,black] (5,0.68526) circle (1.5pt);   
\draw[fill,black] (5.3125,0.65947) circle (1.5pt);    
\draw[fill,black] (5.625,0.65579) circle (1.5pt);     
\draw[fill,black] (5.9375,0.67421) circle (1.5pt);    
\draw[fill,black] (6.25,0.71474) circle (1.5pt);   
\draw[fill,black] (6.5625,0.73316) circle (1.5pt);    
\draw[fill,black] (6.875,0.77368) circle (1.5pt);     
\draw[fill,black] (7.1875,0.84737) circle (1.5pt);    
\draw[fill,black] (7.5,0.96895) circle (1.5pt);    
\draw[fill,black] (7.8125,1.1495) circle (1.5pt);     
\draw[fill,black] (8.125,1.4995) circle (1.5pt);   
\draw[fill,black] (8.4375,2.0042) circle (1.5pt);     
\draw[fill,black] (8.75,3.08) circle (1.5pt);   
\draw[fill,black] (9.0625,5.0879) circle (1.5pt);     
\draw[fill,black] (9.375,6.1489) circle (1.5pt);   
\draw[fill,black] (9.6875,5.8026) circle (1.5pt);     
\draw[fill,black] (10,4.0453) circle (1.5pt);   
\draw[fill,black] (10.3125,2.6195) circle (1.5pt);    
\draw[fill,black] (10.625,1.5326) circle (1.5pt);     
\draw[fill,black] (10.9375,0.90632) circle (1.5pt);   
\draw[fill,black] (11.25,0.63368) circle (1.5pt);     
\draw[fill,black] (11.5625,0.46789) circle (1.5pt);   
\draw[fill,black] (11.875,0.42368) circle (1.5pt);    
\draw[fill,black] (12.1875,0.39421) circle (1.5pt);   
\draw[fill,black] (12.5,0.39053) circle (1.5pt);   
\draw[fill,black] (12.8125,0.40158) circle (1.5pt);   
\draw[fill,black] (13.125,0.37947) circle (1.5pt);    
\draw[fill,black] (13.4375,0.37211) circle (1.5pt);   
\draw[fill,black] (13.75,0.37947) circle (1.5pt);     
\draw[fill,black] (14.0625,0.37211) circle (1.5pt);   
\draw[fill,black] (14.375,0.40895) circle (1.5pt);    
\draw[fill,black] (14.6875,0.37211) circle (1.5pt);   
\draw[fill,black] (15,0.39789) circle (1.5pt);     
\draw[thick,dashed,black] (0,0.39457)--(8.234,0.97352)--(9.4635,6.6557)--(10.7904,0.61132)--(15,0.29958);
      
\draw[thick,black] (0,0.47282)--(7.3887,0.7641)--(8.4955,1.988)--(9.4782,6.8134)--(10.4879,1.6709);
\draw[thick,black] (10.4879,1.6709)--(11.3544,0.42463)--(15,0.37173);  

\draw[fill,black] (12.9,5.4) circle (1.2pt);  
\draw (13.3,5.4) node[right=2pt] {\small $\big(x_i,f_i\big)$}; 
\draw[thick,dashed] (12.5,5.0)--(13.3,5.0) node[right=2pt] {\small $k=3$}; 
\draw[thick] (12.5,4.6)--(13.3,4.6) node[right=2pt] {\small $k=5$};  
\draw[very thin] (12.3,4.3)--(14.8,4.3)--(14.8,5.8)--(12.3,5.8)--(12.3,4.3);
  
\draw (0,-1) node[right=0pt]{\makebox[15cm][c]{\small Temperature}}; 
\draw (-1.2,0) node[above=0pt]{\rotatebox{90}{\makebox[6cm][c]{\small Titanium Heat Data}}};
\end{tikzpicture}
} 
\end{center}
\caption{Best approximation of the Titanium Heat Data by first-degree splines with $k=3$ bzw. $k=5$ free knots respectively.}
\label{fig4}
\end{figure}
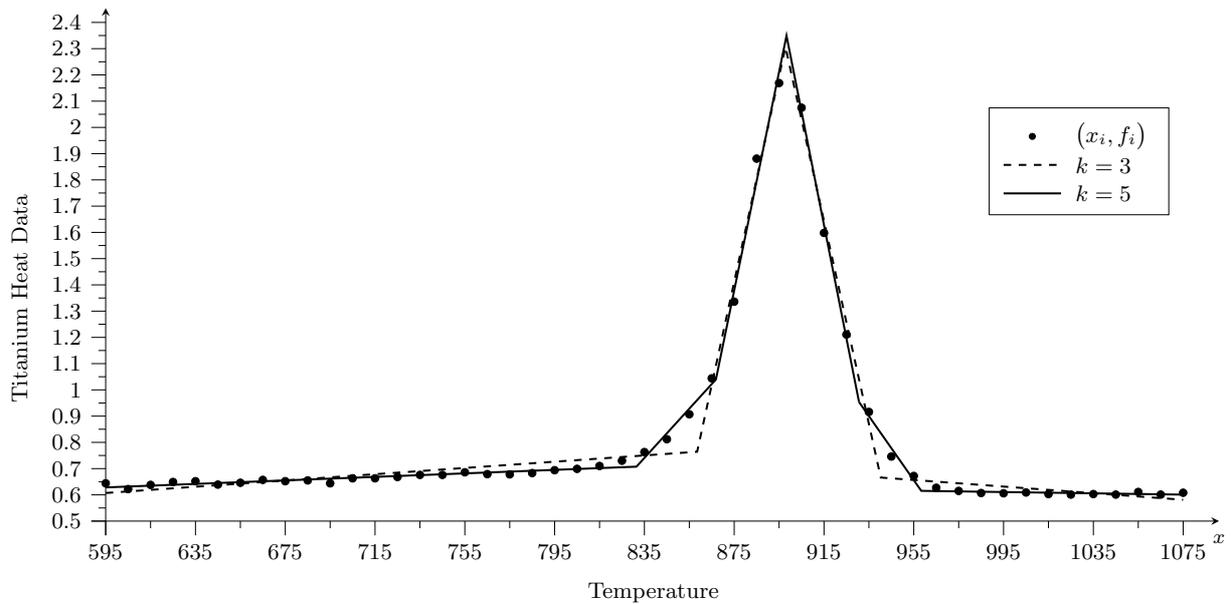

\vspace{0.2cm}
\begin{exmp}\label{allgBsp3}
We look at the equidistant data abscissae $x_i = -1 + \frac{i}{10}$, $i=0,1,\ldots,20$ and the convex values $f_i := x_i^2$, $i=0,1,\ldots,20$. Approximation of the data $(x_i,f_i)$ by a first-order spline with $k=5$ free knots preserves the convexity of the data as can be seen in Fig. \ref{fig5}. 
\end{exmp}

\begin{figure}[h]
\begin{center}
\resizebox{0.75\textwidth}{!}{
\begin{tikzpicture} 
\draw[->,color=black,>=stealth] (-0.2,0) -- (12.5,0) node[below=2pt] {\footnotesize $x$};
\draw[->,color=black,>=stealth] (0,-0.2) -- (0,7.2); 
\draw (0,0-.2)--(0,0) node[below=6pt]{\small -1}; 
\draw (0.6,0-.1)--(0.6,0);
\draw (1.2,0-.2)--(1.2,0) node[below=6pt]{\small -0.8}; 
\draw (1.8,0-.1)--(1.8,0);
\draw (2.4,0-.2)--(2.4,0) node[below=6pt]{\small -0.6}; 
\draw (3,0-.1)--(3,0); 
\draw (3.6,0-.2)--(3.6,0) node[below=6pt]{\small -0.4}; 
\draw (4.2,0-.1)--(4.2,0);
\draw (4.8,0-.2)--(4.8,0) node[below=6pt]{\small -0.2}; 
\draw (5.4,0-.1)--(5.4,0);
\draw (6,0-.2)--(6,0) node[below=6pt]{\small 0};  
\draw (6.6,0-.1)--(6.6,0);
\draw (7.2,0-.2)--(7.2,0) node[below=6pt]{\small 0.2};  
\draw (7.8,0-.1)--(7.8,0);
\draw (8.4,0-.2)--(8.4,0) node[below=6pt]{\small 0.4};  
\draw (9,0-.1)--(9,0); 
\draw (9.6,0-.2)--(9.6,0) node[below=6pt]{\small 0.6};  
\draw (10.2,0-.1)--(10.2,0); 
\draw (10.8,0-.2)--(10.8,0) node[below=6pt]{\small 0.8};
\draw (11.4,0-.1)--(11.4,0); 
\draw (12,0-.2)--(12,0) node[below=6pt]{\small 1};
\draw (0-.2,0)--(0,0) node[left=6pt]{\small -0.1};
\draw (0-.1,0.31818)--(0,0.31818); 
\draw (0-.2,0.63636)--(0,0.63636) node[left=6pt]{\small 0};
\draw (0-.1,0.95455)--(0,0.95455); 
\draw (0-.2,1.2727)--(0,1.2727) node[left=6pt]{\small 0.1};
\draw (0-.1,1.5909)--(0,1.5909);
\draw (0-.2,1.9091)--(0,1.9091) node[left=6pt]{\small 0.2};
\draw (0-.1,2.2273)--(0,2.2273);
\draw (0-.2,2.5455)--(0,2.5455) node[left=6pt]{\small 0.3};
\draw (0-.1,2.8636)--(0,2.8636);
\draw (0-.2,3.1818)--(0,3.1818) node[left=6pt]{\small 0.4};
\draw (0-.1,3.5)--(0,3.5);
\draw (0-.2,3.8182)--(0,3.8182) node[left=6pt]{\small 0.5};
\draw (0-.1,4.1364)--(0,4.1364);
\draw (0-.2,4.4545)--(0,4.4545) node[left=6pt]{\small 0.6};
\draw (0-.1,4.7727)--(0,4.7727);
\draw (0-.2,5.0909)--(0,5.0909) node[left=6pt]{\small 0.7};
\draw (0-.1,5.4091)--(0,5.4091);
\draw (0-.2,5.7273)--(0,5.7273) node[left=6pt]{\small 0.8};
\draw (0-.1,6.0455)--(0,6.0455);
\draw (0-.2,6.3636)--(0,6.3636) node[left=6pt]{\small 0.9};
\draw (0-.1,6.6818)--(0,6.6818);
\draw (0-.2,7)--(0,7) node[left=6pt]{\small 1};
\color{black} 
\draw[fill,black] (0,7) circle (2pt); 
\draw[fill,black] (0.6,5.7909) circle (2pt);
\draw[fill,black] (1.2,4.7091) circle (2pt);
\draw[fill,black] (1.8,3.7545) circle (2pt);
\draw[fill,black] (2.4,2.9273) circle (2pt);
\draw[fill,black] (3,2.2273) circle (2pt);  
\draw[fill,black] (3.6,1.6545) circle (2pt);
\draw[fill,black] (4.2,1.2091) circle (2pt);
\draw[fill,black] (4.8,0.89091) circle (2pt);  
\draw[fill,black] (5.4,0.7) circle (2pt);
\draw[fill,black] (6,0.63636) circle (2pt); 
\draw[fill,black] (6.6,0.7) circle (2pt);
\draw[fill,black] (7.2,0.89091) circle (2pt);  
\draw[fill,black] (7.8,1.2091) circle (2pt);
\draw[fill,black] (8.4,1.6545) circle (2pt);
\draw[fill,black] (9,2.2273) circle (2pt);  
\draw[fill,black] (9.6,2.9273) circle (2pt);
\draw[fill,black] (10.2,3.7545) circle (2pt);  
\draw[fill,black] (10.8,4.7091) circle (2pt);  
\draw[fill,black] (11.4,5.7909) circle (2pt);  
\draw[fill,black] (12,7) circle (2pt);
\draw[thick,black] (0,6.9788)--(1.6,3.9242)--(3.9,1.2409)--(6.3,0.47727)--(8.6,1.697);
\draw[thick,black] (8.6,1.697)--(10.5,4.1152)--(12,6.9788);
  
\end{tikzpicture} 
} 
\end{center}
\caption{The best approximation to convex data is convex again in Example \ref{allgBsp3}.}
\label{fig5}
\end{figure}
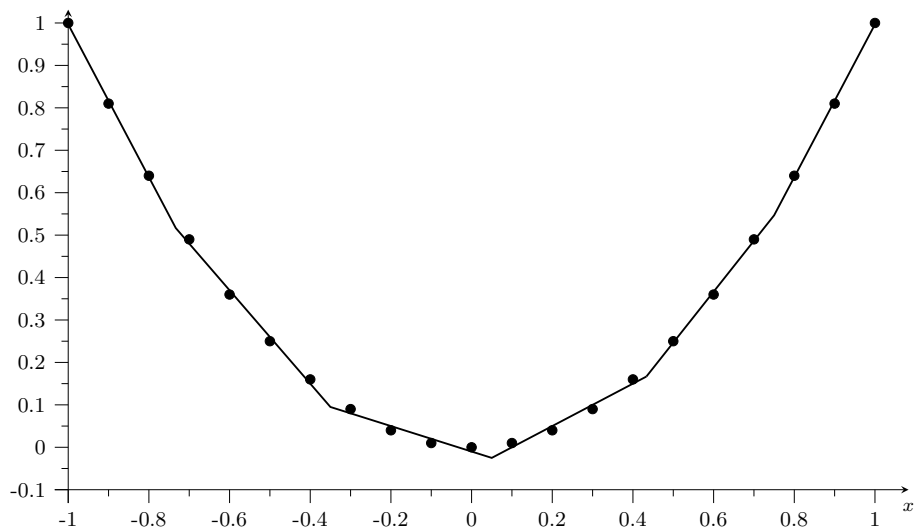

\vspace{0.2cm}

\vspace{0.2in}
\begin{exmp}\label{allgBsp4}
For $k=2,3$ the data $(x_i,f_i)$ from Fig. \ref{fig6} were approximated by broken lines from $S^1_k\left[x_0,x_{71}\right]$. The values $f_i$ are the aritmetic mean of measurements of forces measured in approximately equidistant time intervals when a sub-way switch ist set.  

\vspace{0.05in}
For $k=2$ a best approximation $s_2$ was computed with the (interior) knots $t_1 = 0.77277$ and $t_2 = 1.83478$ and approximation error $\left\|F-s_2(X)\right\|_2 = 686.36122$. Raising the number of free knots to $k=3$ lowers the approximation error to $\left\|F-s_3(X)\right\|_2 = 460.5584$ where the computed best approximation $s_3$ has the three (interior) knots $t_1 = 0.25851, t_2 = 0.87598$ and $t_3 = 1.81896$.
\end{exmp} 

\begin{figure}[h]
\begin{center}
\scalebox{0.75}{
\begin{picture}(415,240)
\color{black}
\put(35,35){\vector(1,0){375}}
\put(40,30){\vector(0,1){203}}
\put(37,19){\small $0$}
\put(50,33.5){\line(0,1){3}}
\put(60,33.5){\line(0,1){3}}
\put(70,33.5){\line(0,1){3}}
\put(80,33.5){\line(0,1){3}}
\put(90,32){\line(0,1){6}}
\put(84,19){\small $0.5$}
\put(100,33.5){\line(0,1){3}}
\put(110,33.5){\line(0,1){3}}
\put(120,33.5){\line(0,1){3}}
\put(130,33.5){\line(0,1){3}}
\put(140,30){\line(0,1){10}}
\put(137,19){\small $1$}
\put(150,33.5){\line(0,1){3}}
\put(160,33.5){\line(0,1){3}}
\put(170,33.5){\line(0,1){3}}
\put(180,33.5){\line(0,1){3}}
\put(190,32){\line(0,1){6}}
\put(184,19){\small $1.5$}
\put(200,33.5){\line(0,1){3}}
\put(210,33.5){\line(0,1){3}}
\put(220,33.5){\line(0,1){3}}
\put(230,33.5){\line(0,1){3}}
\put(240,30){\line(0,1){10}}
\put(237,19){\small $2$}
\put(250,33.5){\line(0,1){3}}
\put(260,33.5){\line(0,1){3}}
\put(270,33.5){\line(0,1){3}}
\put(280,33.5){\line(0,1){3}}
\put(290,32){\line(0,1){6}}
\put(284,19){\small $2.5$}
\put(300,33.5){\line(0,1){3}}
\put(310,33.5){\line(0,1){3}}
\put(320,33.5){\line(0,1){3}}
\put(330,33.5){\line(0,1){3}}
\put(340,30){\line(0,1){10}}
\put(337,19){\small $3$}
\put(350,33.5){\line(0,1){3}}
\put(360,33.5){\line(0,1){3}}
\put(370,33.5){\line(0,1){3}}
\put(380,33.5){\line(0,1){3}}
\put(390,32){\line(0,1){6}}
\put(384,19){\small $3.5$}
\put(400,33.5){\line(0,1){3}}
\put(407,19){\small $t$}
\put(27,33){\footnotesize $0$}
\put(36,50.2928){\line(1,0){8}}
\put(38,42.6464){\line(1,0){4}}
\put(19,48.2928){\footnotesize $144$}
\put(36,65.5855){\line(1,0){8}}
\put(38,57.9392){\line(1,0){4}}
\put(19,63.5855){\footnotesize $288$}
\put(36,80.8783){\line(1,0){8}}
\put(38,73.2319){\line(1,0){4}}
\put(19,78.8783){\footnotesize $432$}
\put(36,96.1711){\line(1,0){8}}
\put(38,88.5247){\line(1,0){4}}
\put(19,94.1711){\footnotesize $576$}
\put(36,111.4638){\line(1,0){8}}
\put(38,103.8175){\line(1,0){4}}
\put(19,109.4638){\footnotesize $720$}
\put(36,126.7566){\line(1,0){8}}
\put(38,119.1102){\line(1,0){4}}
\put(19,124.7566){\footnotesize $864$}
\put(36,142.0494){\line(1,0){8}}
\put(38,134.403){\line(1,0){4}}
\put(15,140.0494){\footnotesize $1008$}
\put(36,157.3421){\line(1,0){8}}
\put(38,149.6958){\line(1,0){4}}
\put(15,155.3421){\footnotesize $1152$}
\put(36,172.6349){\line(1,0){8}}
\put(38,164.9885){\line(1,0){4}}
\put(15,170.6349){\footnotesize $1296$}
\put(36,187.9277){\line(1,0){8}}
\put(38,180.2813){\line(1,0){4}}
\put(15,185.9277){\footnotesize $1440$}
\put(36,203.2204){\line(1,0){8}}
\put(38,195.5741){\line(1,0){4}}
\put(15,201.2204){\footnotesize $1584$}
\put(36,218.5132){\line(1,0){8}}
\put(38,210.8668){\line(1,0){4}}
\put(36,220){\line(1,0){8}}
\put(15,218){\footnotesize $1742$}
\color{black}
\put(40,88.3573){\circle*{4}}
\put(45,86.4627){\circle*{4}}
\put(50,101.3533){\circle*{4}}
\put(55,138.2756){\circle*{4}}
\put(60,148.7197){\circle*{4}}
\put(65,160.5468){\circle*{4}}
\put(70,172.5405){\circle*{4}}
\put(75,179.181){\circle*{4}}
\put(80,175.859){\circle*{4}}
\put(85,177.1325){\circle*{4}}
\put(90,169.9366){\circle*{4}}
\put(95,172.646){\circle*{4}}
\put(100,184.5266){\circle*{4}}
\put(105,198.3454){\circle*{4}}
\put(110,209.0049){\circle*{4}}
\put(115,206.519){\circle*{4}}
\put(120,199.3861){\circle*{4}}
\put(125,204.573){\circle*{4}}
\put(130,205.5291){\circle*{4}}
\put(135,195.8562){\circle*{4}}
\put(140,174.8539){\circle*{4}}
\put(145,171.2178){\circle*{4}}
\put(150,161.4353){\circle*{4}}
\put(155,167.7749){\circle*{4}}
\put(160,169.1389){\circle*{4}}
\put(165,162.5369){\circle*{4}}
\put(170,153.2783){\circle*{4}}
\put(175,138.2067){\circle*{4}}
\put(180,124.3857){\circle*{4}}
\put(185,122.3476){\circle*{4}}
\put(190,108.1598){\circle*{4}}
\put(195,86.5157){\circle*{4}}
\put(200,85.8984){\circle*{4}}
\put(205,82.3306){\circle*{4}}
\put(210,76.1772){\circle*{4}}
\put(215,64.5811){\circle*{4}}
\put(220,56.3554){\circle*{4}}
\put(225,53.0719){\circle*{4}}
\put(230,54.2996){\circle*{4}}
\put(235,56.903){\circle*{4}}
\put(240,55.8285){\circle*{4}}
\put(245,54.6657){\circle*{4}}
\put(250,54.2186){\circle*{4}}
\put(255,54.922){\circle*{4}}
\put(260,57.8817){\circle*{4}}
\put(265,59.034){\circle*{4}}
\put(270,59.2124){\circle*{4}}
\put(275,59.3007){\circle*{4}}
\put(280,59.2954){\circle*{4}}
\put(285,59.2941){\circle*{4}}
\put(290,59.3116){\circle*{4}}
\put(295,59.3678){\circle*{4}}
\put(300,59.3905){\circle*{4}}
\put(305,59.5098){\circle*{4}}
\put(310,59.6261){\circle*{4}}
\put(315,59.6048){\circle*{4}}
\put(320,59.5709){\circle*{4}}
\put(325,59.5629){\circle*{4}}
\put(330,59.5885){\circle*{4}}
\put(335,59.563){\circle*{4}}
\put(340,59.5935){\circle*{4}}
\put(345,59.5499){\circle*{4}}
\put(350,59.5904){\circle*{4}}
\put(355,59.6194){\circle*{4}}
\put(360,59.5986){\circle*{4}}
\put(365,59.6113){\circle*{4}}
\put(370,59.6466){\circle*{4}}
\put(375,59.5776){\circle*{4}}
\put(380,59.5952){\circle*{4}}
\put(385,59.592){\circle*{4}}
\put(390,59.6352){\circle*{4}}
\put(395,59.6112){\circle*{4}}
\thicklines
\drawline[-35](40,104.6299)(117.277,219.9438)(223.4785,56.0848)(395,60.9149)
\thinlines
\dottedline[\circle*{1}](40,78.8592)(65.8507,165.2214)(127.5977,208.6296)(221.8959,56.0402)(395,60.9149)
\color{black}
\put(185,4){\small Time (seconds)}
\put(2,87){\rotatebox{90}{\small Force (Newton)}}
\put(340,220){\circle*{4}}
\thicklines
\drawline[-30](320,205)(360,205)
\thinlines
\dottedline[\circle*{1}](320,190)(360,190)
\color{black}
\put(368,217){\small $\left(x_i,f_i\right)$}
\put(371,202){\small $k=2$}
\put(371,187){\small $k=3$}
\end{picture}
} 
\end{center}
\caption{Best approximation from Example \ref{allgBsp4}.}
\label{fig6}
\end{figure}
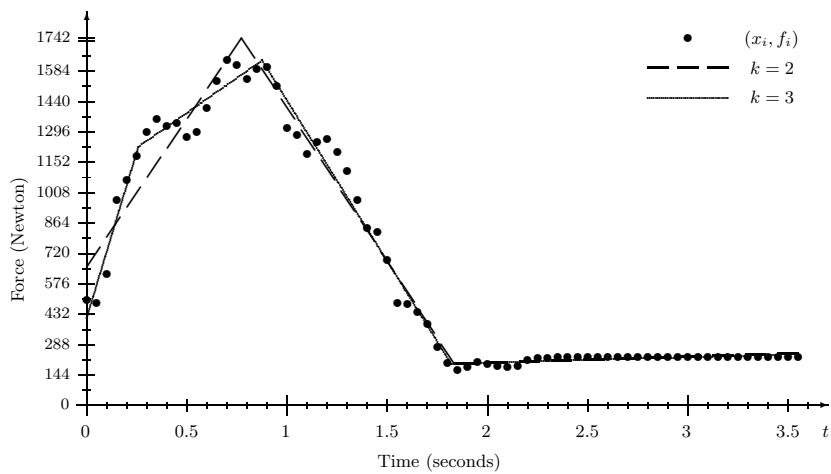

\vspace{0.2cm}
\begin{exmp}\label{allgBsp5}
Fig. \ref{fig7} shows the best approximations from $S_7^1[x_0,x_{19}]$ and from $S_{10}^1[x_0,x_{19}]$ to $\mu+2=20$ data points $(x_i,f_i)$ with $x_i := i+1$, $i=0,1,2,\ldots,20$. In computing the best approximation in $S_7^1[x_0,x_{19}]$ Algorithm \ref{eq:basisalgorithmus} inspected $6\,724\,520$ position vectors with $1 \leq p(1) < p(2) < \ldots < p(10) \leq 2\mu-1$, only  $795\,455$ of which are regular. For the approximation from $S_{10}^1[x_0,x_{19}]$ the algorithm inspected $183\,579\,396$ position vectors with $1 \leq p(1) < p(2) < \ldots < p(10) \leq 2\mu-1$, of which $1\,256\,465$ were regular.
\end{exmp}

\begin{figure}[h]
\begin{center}
\resizebox{0.75\textwidth}{!}{
\begin{tikzpicture}
\draw[->,color=black,>=stealth] (-0.2,0) -- (12.5,0) node[below=2pt] {\footnotesize $x$};  
\draw[->,color=black,>=stealth] (0,-0.2) -- (0,7.2);   
\draw (0,0-.2)--(0,0) node[below=6pt]{\small 1};   
\draw (0.31579,0-.1)--(0.31579,0); 
\draw (0.63158,0-.2)--(0.63158,0) node[below=6pt]{\small 2};   
\draw (0.94737,0-.1)--(0.94737,0); 
\draw (1.2632,0-.2)--(1.2632,0) node[below=6pt]{\small 3}; 
\draw (1.5789,0-.1)--(1.5789,0);   
\draw (1.8947,0-.2)--(1.8947,0) node[below=6pt]{\small 4}; 
\draw (2.2105,0-.1)--(2.2105,0);   
\draw (2.5263,0-.2)--(2.5263,0) node[below=6pt]{\small 5}; 
\draw (2.8421,0-.1)--(2.8421,0);   
\draw (3.1579,0-.2)--(3.1579,0) node[below=6pt]{\small 6}; 
\draw (3.4737,0-.1)--(3.4737,0);   
\draw (3.7895,0-.2)--(3.7895,0) node[below=6pt]{\small 7}; 
\draw (4.1053,0-.1)--(4.1053,0);   
\draw (4.4211,0-.2)--(4.4211,0) node[below=6pt]{\small 8}; 
\draw (4.7368,0-.1)--(4.7368,0);   
\draw (5.0526,0-.2)--(5.0526,0) node[below=6pt]{\small 9}; 
\draw (5.3684,0-.1)--(5.3684,0);   
\draw (5.6842,0-.2)--(5.6842,0) node[below=6pt]{\small 10};
\draw (6,0-.1)--(6,0); 
\draw (6.3158,0-.2)--(6.3158,0) node[below=6pt]{\small 11};
\draw (6.6316,0-.1)--(6.6316,0);   
\draw (6.9474,0-.2)--(6.9474,0) node[below=6pt]{\small 12};
\draw (7.2632,0-.1)--(7.2632,0);   
\draw (7.5789,0-.2)--(7.5789,0) node[below=6pt]{\small 13};
\draw (7.8947,0-.1)--(7.8947,0);   
\draw (8.2105,0-.2)--(8.2105,0) node[below=6pt]{\small 14};
\draw (8.5263,0-.1)--(8.5263,0);   
\draw (8.8421,0-.2)--(8.8421,0) node[below=6pt]{\small 15};
\draw (9.1579,0-.1)--(9.1579,0);   
\draw (9.4737,0-.2)--(9.4737,0) node[below=6pt]{\small 16};
\draw (9.7895,0-.1)--(9.7895,0);   
\draw (10.1053,0-.2)--(10.1053,0) node[below=6pt]{\small 17};  
\draw (10.4211,0-.1)--(10.4211,0); 
\draw (10.7368,0-.2)--(10.7368,0) node[below=6pt]{\small 18};  
\draw (11.0526,0-.1)--(11.0526,0); 
\draw (11.3684,0-.2)--(11.3684,0) node[below=6pt]{\small 19};  
\draw (11.6842,0-.1)--(11.6842,0); 
\draw (12,0-.2)--(12,0) node[below=6pt]{\small 20};
\draw (0-.2,0)--(0,0) node[left=6pt]{\small 0};
\draw (0-.1,0.26923)--(0,0.26923); 
\draw (0-.2,0.53846)--(0,0.53846) node[left=6pt]{\small 0.1};  
\draw (0-.1,0.80769)--(0,0.80769); 
\draw (0-.2,1.0769)--(0,1.0769) node[left=6pt]{\small 0.2};
\draw (0-.1,1.3462)--(0,1.3462);   
\draw (0-.2,1.6154)--(0,1.6154) node[left=6pt]{\small 0.3};
\draw (0-.1,1.8846)--(0,1.8846);   
\draw (0-.2,2.1538)--(0,2.1538) node[left=6pt]{\small 0.4};
\draw (0-.1,2.4231)--(0,2.4231);   
\draw (0-.2,2.6923)--(0,2.6923) node[left=6pt]{\small 0.5};
\draw (0-.1,2.9615)--(0,2.9615);   
\draw (0-.2,3.2308)--(0,3.2308) node[left=6pt]{\small 0.6};
\draw (0-.1,3.5)--(0,3.5); 
\draw (0-.2,3.7692)--(0,3.7692) node[left=6pt]{\small 0.7};
\draw (0-.1,4.0385)--(0,4.0385);   
\draw (0-.2,4.3077)--(0,4.3077) node[left=6pt]{\small 0.8};
\draw (0-.1,4.5769)--(0,4.5769);   
\draw (0-.2,4.8462)--(0,4.8462) node[left=6pt]{\small 0.9};
\draw (0-.1,5.1154)--(0,5.1154);   
\draw (0-.2,5.3846)--(0,5.3846) node[left=6pt]{\small 1};  
\draw (0-.1,5.6538)--(0,5.6538);   
\draw (0-.2,5.9231)--(0,5.9231) node[left=6pt]{\small 1.1};
\draw (0-.1,6.1923)--(0,6.1923);   
\draw (0-.2,6.4615)--(0,6.4615) node[left=6pt]{\small 1.2};
\draw (0-.1,6.7308)--(0,6.7308);   
\draw (0-.2,7)--(0,7) node[left=6pt]{\small 1.3};  
\color{black}  
\draw[fill,black] (0,4.5932) circle (2pt); 
\draw[fill,black] (0.63158,3.3495) circle (2pt);   
\draw[fill,black] (1.2632,1.8897) circle (2pt);
\draw[fill,black] (1.8947,2.7637) circle (2pt);
\draw[fill,black] (2.5263,2.1636) circle (2pt);
\draw[fill,black] (3.1579,0.40905) circle (2pt);   
\draw[fill,black] (3.7895,1.2919) circle (2pt);
\draw[fill,black] (4.4211,0.66403) circle (2pt);   
\draw[fill,black] (5.0526,0.99027) circle (2pt);   
\draw[fill,black] (5.6842,1.2921) circle (2pt);
\draw[fill,black] (6.3158,2.2468) circle (2pt);
\draw[fill,black] (6.9474,0.26737) circle (2pt);   
\draw[fill,black] (7.5789,4.8608) circle (2pt);
\draw[fill,black] (8.2105,5.0873) circle (2pt);
\draw[fill,black] (8.8421,2.6431) circle (2pt);
\draw[fill,black] (9.4737,2.6344) circle (2pt);
\draw[fill,black] (10.1053,1.8185) circle (2pt);   
\draw[fill,black] (10.7368,4.8464) circle (2pt);   
\draw[fill,black] (11.3684,1.9883) circle (2pt);   
\draw[fill,black] (12,0.59878) circle (2pt);   
\draw[thick,black] (0,4.2142)--(3.3664,0.61941)--(6.3158,1.8046)--(6.9474,0.26737)--(7.811,6.5484)--(8.8421,2.7777)--(10.1053,1.953)--(10.7368,4.6017)--(12,0.354); 
\draw[dashed,thick,black] (0,4.6293)--(1.2632,1.9258)--(2.1365,3.0844)--(3.1579,0.67143)--(5.5537,1.0947)--(6.3158,2.2468)--(6.9474,0.26737)--(7.811,6.5484)--(8.8421,2.7777)--(10.1053,1.953)--(10.7368,4.6017)--(12,0.354);

\draw[fill,black] (3.9,6.4) circle (2pt);  
\draw (4.3,6.4) node[right=2pt] {\small $\big(x_i,f_i\big)$}; 
\draw[thick] (3.5,6.0)--(4.3,6.0) node[right=2pt] {\small $s_7$}; 
\draw[thick,dashed] (3.5,5.6)--(4.3,5.6) node[right=2pt] {\small $s_{10}$};  
\draw[very thin] (3.3,5.3)--(5.8,5.3)--(5.8,6.8)--(3.3,6.8)--(3.3,5.3);
   
\end{tikzpicture}
} 
\end{center}
\caption{Best approximations $s_7\in S_7^1[1,20]$ and $s_{10}\in S_{10}^1[1,20]$ to the $20$ data points $(x_i,f_i)$ from Example \ref{allgBsp5}.}
\label{fig7}
\end{figure}
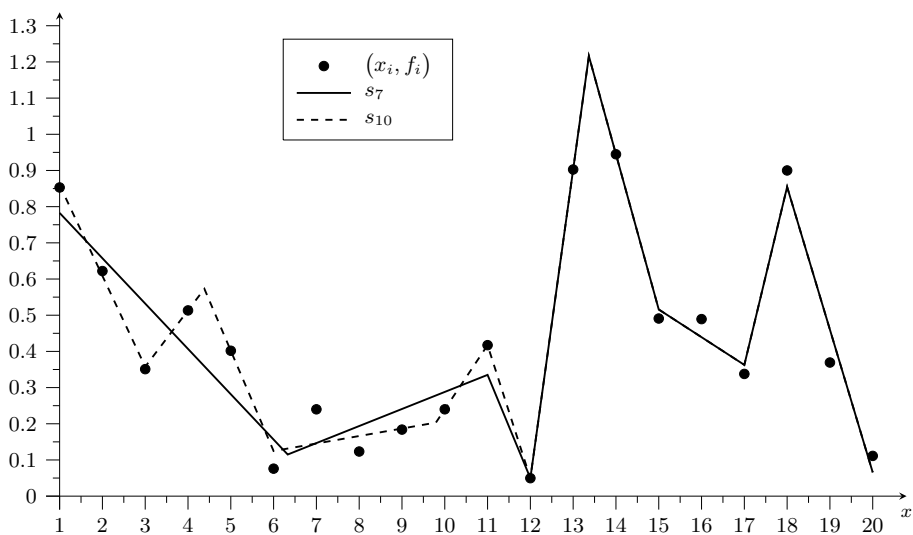

\vspace{0.2cm}
Next, we present a medical application which originally motivated the development of Algorithm \ref{eq:basisalgorithmus}. The medical objective is to determine with an automated procedure the lowest concentration of an antibiotic necessary to kill (a defined high proportion of) the targeted bacteria, the so-called \textit{\underline{m}inimal \underline{b}actericidal \underline{c}oncentration} (MBC). In clinical applications this helps to fight germs effectictely whith minimal unwanted side effects. 

\vspace{0.05in}
The measurements are determined as follows: An inital concentration $\kappa_0$ of the antibiotic is halved in each step. Germs are then exposed to each dilution $\kappa(j) := \kappa_0 \cdot 2^{-j}$, $j=0,\ldots,z$,  $z \in \mathbb{N}\setminus\{\,0\,\}$. The effect of concentration $\kappa(j)$ on the viability of the germs ist measured by a flourescence method which allows to derive the values
\begin{equation*}
\varphi_j := \frac{\text{number of surviving germs after exposition}}{\text{total number of germs before exposition}}
\end{equation*}
for $j=0,\ldots,z$. 

\vspace{0.05in}
This gives us $z+1$ data points $\left(\kappa(0),\varphi_0\right), \ldots, \left(\kappa(z),\varphi_{z}\right)$. For better representation instead of concentration $\kappa(j)$ the logarithmic values $\log\left(\kappa(j)\right)$ are used, that ist, we make use of the relation $j \;=\; -\frac{\log(\kappa(j)) - \log(k_0)}{\log(2)}$ and plot the pairs 
\begin{equation*}
\left(0,\varphi_0\right) , \left(1,\varphi_1\right) , \ldots \;,\; \left(z,\varphi_{z}\right). 
\end{equation*}
See Fig. \ref{fig8} for a typical graph of the data and a best approximating first-degree spline $s^{\ast} \in S^1_2\left[x_0,x_z\right]$ with knots $t_1$ und $t_2$. The knot $t_1$ is interpreted from the medical perspective as  \textit{minimal bactericidal concentration}
\begin{equation*}
\text{MBC} \;:=\; \kappa\left(t_1\right) \;=\; \kappa_0 \cdot 2^{-t_1} 
\end{equation*}
Also of medical interest is knot $t_2$ from which
\begin{equation*}
\text{MIC} \;:=\; \kappa\left(t_2\right) \;=\; \kappa_0 \cdot 2^{-t_2},
\end{equation*}
is derived, the \textit{\underline{m}inimal \underline{i}nhibitory \underline{c}oncentration}, the lowest concentration with noticeable bactericidal effect; see \cite{baerSchumannKrebsCromme} for more details.

\vspace{0.1in}
In the following examples the initial concentration was $\kappa_0 := 256 = 2^8$, when $z=19$ and $\kappa_0 := 128 = 2^7$, when $z=18$. The following numerical example shows typical data and a typical best approximation  $s^{\ast} \in S^1_2\left[x_0,x_z\right]$ with monotonically increasing $s^{\ast}$.


\begin{exmp}\label{medBsp1}
With $x_i := i$ for $i=0,1,\ldots,19$ the following medical values $f_i$ must be approximated:
\begin{flushleft}
\resizebox{1.0\textwidth}{!}{
$\begin{array}{r|rrrrrrrrrrr}                                         
i & 0 & 1 & 2 & 3 & 4 & 5 & 6 & 7 & 8 & 9 & 10 \\                               
\hline                                                                                        
f_i & 3.6273 & 3.381 & 3.0339 & 2.8414 & 2.7507 & 2.9006 & 2.941 & 2.9986 & 3.2127 & 3.8381 & 8.2629
\end{array}$
} 
\end{flushleft}                                                                                                                 
\begin{flushleft}
\resizebox{0.95\textwidth}{!}{
$\begin{array}{r|rrrrrrrrr}                                         
i &  11 & 12 & 13 & 14 & 15 & 16 & 17 & 18 & 19 \\                         
\hline                                                                                 
f_i  & 37.7363 & 84.0146 & 94.7914 & 98.7679 & 97.0424 & 98.0432 & 95.5602 & 99.0313 & 100   
\end{array}$
} 
\end{flushleft}

\vspace{0.1in}
\noindent{}As best approximation from  $S^1_2\left[x_0,x_{19}\right]$ Algorithm \ref{eq:basisalgorithmus} computed the broken line
\begin{equation*}                                                
s^{\ast}(x) =                                                    
\left\{                                                                                                    
\begin{array}{lcl}                                                                                         
0.2368 \, x + 2.43313 & , & x < 10.28981 \\                                                                   
46.2783 \, x - 471.325 & , & 10.28981\leq x < 12.25123 \\                                                     
0.52394 \, x + 89.22211 & , & x \geq 12.25123                                                             
\end{array}                                                                                                
\right.                                                          
\end{equation*} 
with the interior knots $t_1 = 10.2898$ and $t_2 = 12.2512$ and the approximation error $\left\|F-s^{\ast}(X)\right\|_2 = 5.7246$; see Fig. \ref{fig8}. This gives us the following values for MIC and MBC
\begin{equation*}
\text{MBC} = 2^{8-t_1} = 0.2045 \hspace{0.3in}\text{and}\hspace{0.3in} \text{MIC} = 2^{8-t_2} = 0.0525.
\end{equation*}
\end{exmp}

\begin{figure}[h]
\begin{center}
\scalebox{0.75}{
\begin{picture}(415,225)
\put(35,35){\vector(1,0){375}}
\put(40,30){\vector(0,1){193}}
\put(37,19){\small $0$}
\put(58.6842,31){\line(0,1){8}}
\put(49.3421,33){\line(0,1){4}}
\put(55.6842,19){\small $1$}
\put(77.3684,31){\line(0,1){8}}
\put(68.0263,33){\line(0,1){4}}
\put(74.3684,19){\small $2$}
\put(96.0526,31){\line(0,1){8}}
\put(86.7105,33){\line(0,1){4}}
\put(93.0526,19){\small $3$}
\put(114.7368,31){\line(0,1){8}}
\put(105.3947,33){\line(0,1){4}}
\put(111.7368,19){\small $4$}
\put(133.4211,31){\line(0,1){8}}
\put(124.0789,33){\line(0,1){4}}
\put(130.4211,19){\small $5$}
\put(152.1053,31){\line(0,1){8}}
\put(142.7632,33){\line(0,1){4}}
\put(149.1053,19){\small $6$}
\put(170.7895,31){\line(0,1){8}}
\put(161.4474,33){\line(0,1){4}}
\put(167.7895,19){\small $7$}
\put(189.4737,31){\line(0,1){8}}
\put(180.1316,33){\line(0,1){4}}
\put(186.4737,19){\small $8$}
\put(208.1579,31){\line(0,1){8}}
\put(198.8158,33){\line(0,1){4}}
\put(205.1579,19){\small $9$}
\put(226.8421,31){\line(0,1){8}}
\put(217.5,33){\line(0,1){4}}
\put(220.8421,19){\small $10$}
\put(245.5263,31){\line(0,1){8}}
\put(236.1842,33){\line(0,1){4}}
\put(239.5263,19){\small $11$}
\put(264.2105,31){\line(0,1){8}}
\put(254.8684,33){\line(0,1){4}}
\put(258.2105,19){\small $12$}
\put(282.8947,31){\line(0,1){8}}
\put(273.5526,33){\line(0,1){4}}
\put(276.8947,19){\small $13$}
\put(301.5789,31){\line(0,1){8}}
\put(292.2368,33){\line(0,1){4}}
\put(295.5789,19){\small $14$}
\put(320.2632,31){\line(0,1){8}}
\put(310.9211,33){\line(0,1){4}}
\put(314.2632,19){\small $15$}
\put(338.9474,31){\line(0,1){8}}
\put(329.6053,33){\line(0,1){4}}
\put(332.9474,19){\small $16$}
\put(357.6316,31){\line(0,1){8}}
\put(348.2895,33){\line(0,1){4}}
\put(351.6316,19){\small $17$}
\put(376.3158,31){\line(0,1){8}}
\put(366.9737,33){\line(0,1){4}}
\put(370.3158,19){\small $18$}
\put(395,31){\line(0,1){8}}
\put(385.6579,33){\line(0,1){4}}
\put(389,19){\small $19$}
\put(405,19){$x$} 
\put(27,33){\footnotesize $0$}
\put(36,49){\line(1,0){8}}
\put(38,42){\line(1,0){4}}
\put(27,47){\footnotesize $8$}
\put(36,63){\line(1,0){8}}
\put(38,56){\line(1,0){4}}
\put(23,61){\footnotesize $16$}
\put(36,77){\line(1,0){8}}
\put(38,70){\line(1,0){4}}
\put(23,75){\footnotesize $24$}
\put(36,91){\line(1,0){8}}
\put(38,84){\line(1,0){4}}
\put(23,89){\footnotesize $32$}
\put(36,105){\line(1,0){8}}
\put(38,98){\line(1,0){4}}
\put(23,103){\footnotesize $40$}
\put(36,119){\line(1,0){8}}
\put(38,112){\line(1,0){4}}
\put(23,117){\footnotesize $48$}
\put(36,133){\line(1,0){8}}
\put(38,126){\line(1,0){4}}
\put(23,131){\footnotesize $56$}
\put(36,147){\line(1,0){8}}
\put(38,140){\line(1,0){4}}
\put(23,145){\footnotesize $64$}
\put(36,161){\line(1,0){8}}
\put(38,154){\line(1,0){4}}
\put(23,159){\footnotesize $72$}
\put(36,175){\line(1,0){8}}
\put(38,168){\line(1,0){4}}
\put(23,173){\footnotesize $80$}
\put(36,189){\line(1,0){8}}
\put(38,182){\line(1,0){4}}
\put(23,187){\footnotesize $88$}
\put(36,203){\line(1,0){8}}
\put(38,196){\line(1,0){4}}
\put(23,201){\footnotesize $96$}
\put(36,210){\line(1,0){8}}
\put(19,208){\footnotesize $100$}
\color{black}
\put(40,41.3478){\circle*{4}}
\put(58.6842,40.9168){\circle*{4}}
\put(77.3684,40.3093){\circle*{4}}
\put(96.0526,39.9725){\circle*{4}}
\put(114.7368,39.8137){\circle*{4}}
\put(133.4211,40.0761){\circle*{4}}
\put(152.1053,40.1467){\circle*{4}}
\put(170.7895,40.2476){\circle*{4}}
\put(189.4737,40.6222){\circle*{4}}
\put(208.1579,41.7167){\circle*{4}}
\put(226.8421,49.4601){\circle*{4}}
\put(245.5263,101.0385){\circle*{4}}
\put(264.2105,182.0256){\circle*{4}}
\put(282.8947,200.885){\circle*{4}}
\put(301.5789,207.8438){\circle*{4}}
\put(320.2632,204.8242){\circle*{4}}
\put(338.9474,206.5756){\circle*{4}}
\put(357.6316,202.2303){\circle*{4}}
\put(376.3158,208.3048){\circle*{4}}
\put(395,210){\circle*{4}}
\dottedline[\circle*{1}](40,39.258)(232.2569,43.522)(268.9045,202.3719)(395,208.5598)
\put(165,4){\small Concentration ($2^{8-x} \cdot \text{mg/l}$)}
\put(5,60){\rotatebox{90}{\small Proportion of viable bacteria (\%)}}
\end{picture}
} 
\end{center}
\caption{Best approximation $s^\ast$ from Example \ref{medBsp1}.}
\label{fig8}
\end{figure}

\vspace{0.2cm}
As seen in the above example the two knots  $t_1$ and $t_2$ of a best approximation partition the medical data into three segments. In many cases the values are slightly monotonically increasing in the first and the last segment and ascending rapidly in the middle segment. 

\vspace{0.05in}
It may happen that no measurements are available from the middle segment. The following examples show that not always two (proper interior) knots are needed for a best approximation and the solutions need not to be unique.
\begin{exmp}\label{medBsp2}
For $i=0,1,\ldots,18$ the medical data $\left(x_i,f_i\right)$ are given by $x_i := i$ and the values $f_i$ from the following table:

\vspace{0.1in}
$\begin{array}{r|rrrrrrrrrr}                                         
i & 0 & 1 & 2 & 3 & 4 & 5 & 6 & 7 & 8 & 9 \\
\hline                                          
f_i & 4.8245 & 5.0786 & 5.7781 & 6.105 & 5.9493 & 6.0516 & 5.589 & 5.5087 & 5.2563 & 4.5123 
\end{array}$                                                       
                                                                   
\vspace{0.05in}                                                             
$\begin{array}{r|rrrrrrrrr}                                         
i &  10 & 11 & 12 & 13 & 14 & 15 & 16 & 17 & 18  \\                         
\hline                                                                                 
f_i & 97.8802 & 96.3044 & 95.6139 & 98.974 & 95.9425 & 96.0353 & 97.0482 & 98.5606 & 100
\end{array}$

\vspace{0.1in}
\noindent{}As best approximation from $S^1_2\left[x_0,x_{18}\right]$ Algorithm \ref{eq:basisalgorithmus} computed the first-degree spline
\begin{equation*}                                                
s^{\ast}(x) =                                                    
\left\{                                                          
\begin{array}{lcl}                                               
-0.026416 \, x + 5.58421 & , & x < 9 \\                                                                        
91.01492 \, x - 813.78779 & , & 9\leq x < 10 \\                                                                
0.25296 \, x + 93.83177 & , & x \geq 10                         
\end{array}                                                      
\right.                                                         
\end{equation*} 
with the data knots $t_1 = 9$ and $t_2 = 10$, see Fig. \ref{fig9}. The associated approximation error is $\left\|F-s^{\ast}(X)\right\|_2 = 4.24581$ and we compute $\text{MBC} = 2^{7-t_1} = 0.25$ and $\text{MIC} = 2^{7-t_2} = 0.125$.

\vspace{0.05in}
Algorithm \ref{eq:basisalgorithmus} computed just one best approximation, namely $s^{\ast}$, but there exist infinitely many best approximations to the given data: Each of the functions
\begin{equation*}
s\left(y_1,y_2,x\right) \;:=\; 
\left\{
\begin{array}{c}
s_1\left(x\right) \\
s_2\left(x\right) \\
s_3\left(x\right)
\end{array}
\right. \;:=\;
\left\{                                                                                                  
\begin{array}{lcl}                                                                                       
-0.026416 \, x + 5.58421 & , & x \leq y_1 \\                                                           
\frac{s_3\left(y_2\right)-s_1\left(y_1\right)}{y_2-y_1}\left(x - y_1\right) + s_1\left(y_1\right) & , & y_1 < x < y_2  \\                                              
0.25296 \, x + 93.83177 & , & x \geq y_2 \\                                                          
\end{array}                                                                                              
\right. 
\end{equation*}
is a best approximation from $S^1_2\left[x_0,x_{18}\right]$ to the data $\left(x_0,f_0\right),\ldots,\left(x_{18},f_{18}\right)$ for arbitrary $y_1,y_2$ with $t_1 \leq y_1 < y_2 \leq t_2$. As limit case we get the \textit{discontinuous} functions
\begin{equation*}
s\left(y,x\right) \;:=\; 
\left\{                                                                                                  
\begin{array}{lcl}                                                                                       
-0.026416 \, x + 5.58421 & , & x \leq y \\                                            
0.25296 \, x + 93.83177 & , & x > y                                                          
\end{array}                                                                                              
\right. 
\end{equation*}
with jump discontinuity in $y$ for arbitrary $y \in \left(t_1,t_2\right)$ as best approximations to the given data. 
\end{exmp}

\begin{figure}[h]
\begin{center}
\scalebox{0.75}{
\begin{picture}(415,225)
\put(35,35){\vector(1,0){375}}
\put(40,30){\vector(0,1){193}}
\put(37,19){\small $0$}
\put(59.7222,31){\line(0,1){8}}
\put(49.8611,33){\line(0,1){4}}
\put(56.7222,19){\small $1$}
\put(79.4444,31){\line(0,1){8}}
\put(69.5833,33){\line(0,1){4}}
\put(76.4444,19){\small $2$}
\put(99.1667,31){\line(0,1){8}}
\put(89.3056,33){\line(0,1){4}}
\put(96.1667,19){\small $3$}
\put(118.8889,31){\line(0,1){8}}
\put(109.0278,33){\line(0,1){4}}
\put(115.8889,19){\small $4$}
\put(138.6111,31){\line(0,1){8}}
\put(128.75,33){\line(0,1){4}}
\put(135.6111,19){\small $5$}
\put(158.3333,31){\line(0,1){8}}
\put(148.4722,33){\line(0,1){4}}
\put(155.3333,19){\small $6$}
\put(178.0556,31){\line(0,1){8}}
\put(168.1944,33){\line(0,1){4}}
\put(175.0556,19){\small $7$}
\put(197.7778,31){\line(0,1){8}}
\put(187.9167,33){\line(0,1){4}}
\put(194.7778,19){\small $8$}
\put(217.5,31){\line(0,1){8}}
\put(207.6389,33){\line(0,1){4}}
\put(214.5,19){\small $9$}
\put(237.2222,31){\line(0,1){8}}
\put(227.3611,33){\line(0,1){4}}
\put(231.2222,19){\small $10$}
\put(256.9444,31){\line(0,1){8}}
\put(247.0833,33){\line(0,1){4}}
\put(250.9444,19){\small $11$}
\put(276.6667,31){\line(0,1){8}}
\put(266.8056,33){\line(0,1){4}}
\put(270.6667,19){\small $12$}
\put(296.3889,31){\line(0,1){8}}
\put(286.5278,33){\line(0,1){4}}
\put(290.3889,19){\small $13$}
\put(316.1111,31){\line(0,1){8}}
\put(306.25,33){\line(0,1){4}}
\put(310.1111,19){\small $14$}
\put(335.8333,31){\line(0,1){8}}
\put(325.9722,33){\line(0,1){4}}
\put(329.8333,19){\small $15$}
\put(355.5556,31){\line(0,1){8}}
\put(345.6944,33){\line(0,1){4}}
\put(349.5556,19){\small $16$}
\put(375.2778,31){\line(0,1){8}}
\put(365.4167,33){\line(0,1){4}}
\put(369.2778,19){\small $17$}
\put(395,31){\line(0,1){8}}
\put(385.1389,33){\line(0,1){4}}
\put(389,19){\small $18$}
\put(405,19){$x$}
\put(27,33){\footnotesize $0$}
\put(36,49){\line(1,0){8}}
\put(38,42){\line(1,0){4}}
\put(27,47){\footnotesize $8$}
\put(36,63){\line(1,0){8}}
\put(38,56){\line(1,0){4}}
\put(23,61){\footnotesize $16$}
\put(36,77){\line(1,0){8}}
\put(38,70){\line(1,0){4}}
\put(23,75){\footnotesize $24$}
\put(36,91){\line(1,0){8}}
\put(38,84){\line(1,0){4}}
\put(23,89){\footnotesize $32$}
\put(36,105){\line(1,0){8}}
\put(38,98){\line(1,0){4}}
\put(23,103){\footnotesize $40$}
\put(36,119){\line(1,0){8}}
\put(38,112){\line(1,0){4}}
\put(23,117){\footnotesize $48$}
\put(36,133){\line(1,0){8}}
\put(38,126){\line(1,0){4}}
\put(23,131){\footnotesize $56$}
\put(36,147){\line(1,0){8}}
\put(38,140){\line(1,0){4}}
\put(23,145){\footnotesize $64$}
\put(36,161){\line(1,0){8}}
\put(38,154){\line(1,0){4}}
\put(23,159){\footnotesize $72$}
\put(36,175){\line(1,0){8}}
\put(38,168){\line(1,0){4}}
\put(23,173){\footnotesize $80$}
\put(36,189){\line(1,0){8}}
\put(38,182){\line(1,0){4}}
\put(23,187){\footnotesize $88$}
\put(36,203){\line(1,0){8}}
\put(38,196){\line(1,0){4}}
\put(23,201){\footnotesize $96$}
\put(36,210){\line(1,0){8}}
\put(19,208){\footnotesize $100$}
\color{black}
\put(40,43.4429){\circle*{4}}
\put(59.7222,43.8875){\circle*{4}}
\put(79.4444,45.1117){\circle*{4}}
\put(99.1667,45.6838){\circle*{4}}
\put(118.8889,45.4113){\circle*{4}}
\put(138.6111,45.5903){\circle*{4}}
\put(158.3333,44.7807){\circle*{4}}
\put(178.0556,44.6402){\circle*{4}}
\put(197.7778,44.1985){\circle*{4}}
\put(217.5,42.8965){\circle*{4}}
\put(237.2222,206.2903){\circle*{4}}
\put(256.9444,203.5327){\circle*{4}}
\put(276.6667,202.3243){\circle*{4}}
\put(296.3889,208.2045){\circle*{4}}
\put(316.1111,202.8994){\circle*{4}}
\put(335.8333,203.0618){\circle*{4}}
\put(355.5556,204.8344){\circle*{4}}
\put(375.2778,207.481){\circle*{4}}
\put(395,210){\circle*{4}}
\dottedline[\circle*{1}](40,44.7724)(217.5,44.3563)(237.2222,203.6324)(395,207.1739)
\thicklines
\dottedline{2}(40,44.7724)(225.3889,44.3378)(229.3333,203.4553)(395,207.1739)
\thinlines
\color{black}
\put(165,4){\small Concentration ($2^{7-x} \cdot \text{mg/l}$)}
\put(5,60){\rotatebox{90}{\small Proportion of viable bacteria (\%)}}
\dottedline[\circle*{1}](260,150)(290,150)
\thicklines
\dottedline{2}(260,135)(290,135)   
\thinlines   
\put(298,147){$s^{\ast}$}
\put(298,132){$s(9.4,9.6,x)$} 
\end{picture}
} 
\end{center}
\caption{Best approximations $s^{\ast}$ and $s(9.4,9.6,x)$ from Example \ref{medBsp2}.}
\label{fig9}
\end{figure}
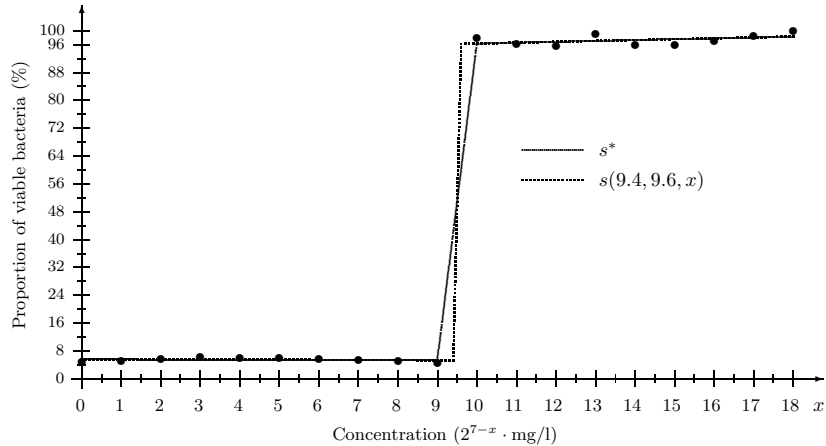

\vspace{0.2cm}
\begin{exmp}\label{medBsp3}
For $i=0,1,\ldots,18$ we look at the data $(x_i,f_i)$ from Example \ref{medBsp2} except that now $f_9 := 7.5123$ and $f_{10} := 97.8802$. As best approximation from $S^1_2\left[x_0,x_{18}\right]$ to the data pairs $\left(x_0,f_0\right),\ldots, \left(x_{18},f_{18}\right)$ thus defined we get from Algorithm \ref{eq:basisalgorithmus} the spline
\begin{equation*}   
s^{\ast}\left(x\right) \;=\;                                              
\left\{                                                                                                   
\begin{array}{lcl}                                                                                        
0.043098 \, x + 5.39884 & , & x < 8.98057 \\                                                                   
88.84909 \, x - 792.12948 & , & 8.98057 \leq x < 10 \\                                                          
0.25296 \, x + 93.83177 & , & x \geq 10                                                           
\end{array}                                                                                               
\right.                                                          
\end{equation*}
with the interior knot $t_1 = 8.98057$, the data knot $t_2 = 10$ and the minimal approximation error $\left\|F_2-s^{\ast}_{2}(X)\right\|_2 = 4.11872$, see Fig. \ref{fig10}. Here again, arbitrarily many solutions  $s\left(y_2,t\right) \in S^1_2\left[t_0,t_{18}\right]$can be derived from $s^{\ast}$ for all of which 
\begin{equation*}
s\left(y_2,x_9\right) = s^{\ast}\left(x_9\right) = f_9 = 7.5123
\end{equation*}
must hold. This implies that for the knots $y_1$ and $y_2$ the inequality $t_1 \leq y_1 < x_9 < y_2 \leq t_2$ must be valid. Because of $s\left(y_2,x_9\right) = f_9$ only one of the two knots can be freely chosen, the other one depends on that choice. If e.g. $y_2$ with $x_9 < y_2 < t_2$ ist chosen, then the splines
\begin{equation*} 
s\left(y_2,t\right) \;:=\; 
\left\{
\begin{array}{c}
s_2\left(t\right) \\
s_2\left(t\right) \\
s_3\left(t\right)
\end{array}
\right. \;:=\;                                          
\left\{                                                                                                   
\begin{array}{lcl}                                                                                        
0.043098 \, x + 5.39884 & , & x \leq y_1 \\                                                                  
\frac{f_9 - s_3(y_2)}{x_9 - y_1}\left(x - x_9\right) + f_9 & , & y_1 < x < y_2 \\
0.25296 \, x + 93.83177 & , & x \geq y_2 \\                                                           
\end{array}                                                                                               
\right.                                                          
\end{equation*}
also are best approximations, where $y_1$ ist determined as solution of $s_1(y_2) = s_2(y_2)$. 
\end{exmp}

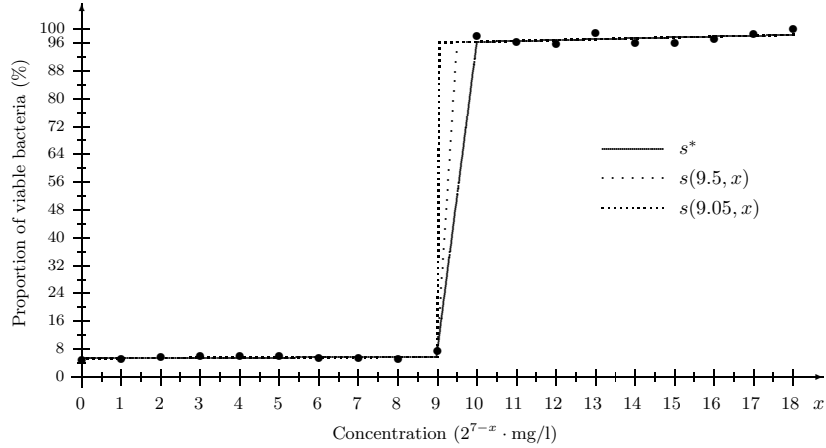
\begin{figure}[h]
\begin{center}
\scalebox{0.75}{
\begin{picture}(415,225)
\put(35,35){\vector(1,0){375}}
\put(40,30){\vector(0,1){193}}
\put(37,19){\small $0$}
\put(59.7222,31){\line(0,1){8}}
\put(49.8611,33){\line(0,1){4}}
\put(56.7222,19){\small $1$}
\put(79.4444,31){\line(0,1){8}}
\put(69.5833,33){\line(0,1){4}}
\put(76.4444,19){\small $2$}
\put(99.1667,31){\line(0,1){8}}
\put(89.3056,33){\line(0,1){4}}
\put(96.1667,19){\small $3$}
\put(118.8889,31){\line(0,1){8}}
\put(109.0278,33){\line(0,1){4}}
\put(115.8889,19){\small $4$}
\put(138.6111,31){\line(0,1){8}}
\put(128.75,33){\line(0,1){4}}
\put(135.6111,19){\small $5$}
\put(158.3333,31){\line(0,1){8}}
\put(148.4722,33){\line(0,1){4}}
\put(155.3333,19){\small $6$}
\put(178.0556,31){\line(0,1){8}}
\put(168.1944,33){\line(0,1){4}}
\put(175.0556,19){\small $7$}
\put(197.7778,31){\line(0,1){8}}
\put(187.9167,33){\line(0,1){4}}
\put(194.7778,19){\small $8$}
\put(217.5,31){\line(0,1){8}}
\put(207.6389,33){\line(0,1){4}}
\put(214.5,19){\small $9$}
\put(237.2222,31){\line(0,1){8}}
\put(227.3611,33){\line(0,1){4}}
\put(231.2222,19){\small $10$}
\put(256.9444,31){\line(0,1){8}}
\put(247.0833,33){\line(0,1){4}}
\put(250.9444,19){\small $11$}
\put(276.6667,31){\line(0,1){8}}
\put(266.8056,33){\line(0,1){4}}
\put(270.6667,19){\small $12$}
\put(296.3889,31){\line(0,1){8}}
\put(286.5278,33){\line(0,1){4}}
\put(290.3889,19){\small $13$}
\put(316.1111,31){\line(0,1){8}}
\put(306.25,33){\line(0,1){4}}
\put(310.1111,19){\small $14$}
\put(335.8333,31){\line(0,1){8}}
\put(325.9722,33){\line(0,1){4}}
\put(329.8333,19){\small $15$}
\put(355.5556,31){\line(0,1){8}}
\put(345.6944,33){\line(0,1){4}}
\put(349.5556,19){\small $16$}
\put(375.2778,31){\line(0,1){8}}
\put(365.4167,33){\line(0,1){4}}
\put(369.2778,19){\small $17$}
\put(395,31){\line(0,1){8}}
\put(385.1389,33){\line(0,1){4}}
\put(389,19){\small $18$}
\put(405,19){$x$}
\put(27,33){\footnotesize $0$}
\put(36,49){\line(1,0){8}}
\put(38,42){\line(1,0){4}}
\put(27,47){\footnotesize $8$}
\put(36,63){\line(1,0){8}}
\put(38,56){\line(1,0){4}}
\put(23,61){\footnotesize $16$}
\put(36,77){\line(1,0){8}}
\put(38,70){\line(1,0){4}}
\put(23,75){\footnotesize $24$}
\put(36,91){\line(1,0){8}}
\put(38,84){\line(1,0){4}}
\put(23,89){\footnotesize $32$}
\put(36,105){\line(1,0){8}}
\put(38,98){\line(1,0){4}}
\put(23,103){\footnotesize $40$}
\put(36,119){\line(1,0){8}}
\put(38,112){\line(1,0){4}}
\put(23,117){\footnotesize $48$}
\put(36,133){\line(1,0){8}}
\put(38,126){\line(1,0){4}}
\put(23,131){\footnotesize $56$}
\put(36,147){\line(1,0){8}}
\put(38,140){\line(1,0){4}}
\put(23,145){\footnotesize $64$}
\put(36,161){\line(1,0){8}}
\put(38,154){\line(1,0){4}}
\put(23,159){\footnotesize $72$}
\put(36,175){\line(1,0){8}}
\put(38,168){\line(1,0){4}}
\put(23,173){\footnotesize $80$}
\put(36,189){\line(1,0){8}}
\put(38,182){\line(1,0){4}}
\put(23,187){\footnotesize $88$}
\put(36,203){\line(1,0){8}}
\put(38,196){\line(1,0){4}}
\put(23,201){\footnotesize $96$}
\put(36,210){\line(1,0){8}}
\put(19,208){\footnotesize $100$}
\color{black}
\put(40,43.4429){\circle*{4}}
\put(59.7222,43.8875){\circle*{4}}
\put(79.4444,45.1117){\circle*{4}}
\put(99.1667,45.6838){\circle*{4}}
\put(118.8889,45.4113){\circle*{4}}
\put(138.6111,45.5903){\circle*{4}}
\put(158.3333,44.7807){\circle*{4}}
\put(178.0556,44.6402){\circle*{4}}
\put(197.7778,44.1985){\circle*{4}}
\put(217.5,48.1465){\circle*{4}}
\put(237.2222,206.2903){\circle*{4}}
\put(256.9444,203.5327){\circle*{4}}
\put(276.6667,202.3243){\circle*{4}}
\put(296.3889,208.2045){\circle*{4}}
\put(316.1111,202.8994){\circle*{4}}
\put(335.8333,203.0618){\circle*{4}}
\put(355.5556,204.8344){\circle*{4}}
\put(375.2778,207.481){\circle*{4}}
\put(395,210){\circle*{4}}
\dottedline[\circle*{1}](40,44.448)(217.1168,45.1253)(237.2222,203.6324)(395,207.1739)
\dottedline[\circle*{1}]{5}(40,44.448)(217.2221,45.1257)(227.3611,203.4111)(395,207.1739)
\thicklines
\dottedline{3}(40,44.448)(217.4722,45.1267)(218.4861,203.2119)(395,207.1739)
\thinlines
\put(165,4){\small Concentration ($2^{7-x} \cdot \text{mg/l}$)}                               
\put(5,60){\rotatebox{90}{\small Proportion of viable bacteria (\%)}}      
\dottedline[\circle*{1}](300,150)(330,150)
\dottedline[\circle*{1}]{5}(300,135)(330,135) 
\thicklines
\dottedline{3}(300,120)(330,120)   
\thinlines   
\put(338,147){$s^{\ast}$}
\put(338,132){$s(9.5,x)$}
\put(338,117){$s(9.05,x)$}
\end{picture}
} 
\end{center}
\caption{The best approximations $s^{\ast}$, $s(9.05,x)$ and $s_2(9.5,x)$ from Example \ref{medBsp3}.}
\label{fig10}
\end{figure}

\vspace{0.2cm}
Measurement data not always meet the pattern of Example \ref{medBsp1}. Whether the interpretation of computed knots als MIC and MBC ist still a good idea must be decided from the medical perspective. In the next example we focus on such a situation.
\begin{exmp}\label{medBsp4}
For $i=0,1,\ldots,19$ the medical data pairs $(x_i,f_i)$ are listed in the following table:

\vspace{0.1in}
$\begin{array}{r|rrrrrrrrrr}                                         
i & 0 & 1 & 2 & 3 & 4 & 5 & 6 & 7 & 8 & 9 \\                                
\hline                                                                                        
f_i & 3.0354 & 3.1654 & 3.0862 & 3.0564 & 2.9804 & 2.9632 & 2.8198 & 3.1239 & 3.0576 & 2.9828
\end{array}$                                                       
                                                                   
\vspace{0.1in}                                                             
$\begin{array}{r|rrrrrrrrrr}                                         
i &  10 & 11 & 12 & 13 & 14 & 15 & 16 & 17 & 18 & 19 \\                         
\hline                                                                                 
f_i  & 3.1498 & 3.5877 & 4.0296 & 6.6481 & 9.829 & 12.1237 & 30.1584 & 70.2245 & 89.7225 & 100
\end{array}$

\vspace{0.1in}
\noindent{}As best approximation from $S^1_2\left[t_0,t_{19}\right]$ Algorithm \ref{eq:basisalgorithmus} computed the spline
\begin{equation*}                                                
s^{\ast}(x) =                                                    
\left\{                                                                                               
\begin{array}{lcl}                                                                                    
0.40661 \, t + 1.30283 & , & t < 15.43646 \\                                                          
40.0661 \, t - 610.8992 & , & 15.43646\leq t < 17.30953 \\                                            
10.2775 \, t - 95.2725 & , & t \geq 17.30953                                                     
\end{array}                                                                                           
\right.                                                         
\end{equation*}                                                  
witch the two (interior) knots $t_1 = 15.43646$ and $t_2 = 17.30953$ and the approximation error $\left\|F-s^{\ast}(X)\right\|_2 = 7.69589$, see Fig. \ref{fig11}. This leads to $\text{MBC} = 2^{8-t_1} = 0.00577$ and $\text{MIC} = 2^{8-t_2} = 0.00158$.
\end{exmp}

\begin{figure}[h]
\begin{center}
\scalebox{0.75}{
\begin{picture}(415,225)
\put(35,35){\vector(1,0){375}}
\put(40,30){\vector(0,1){193}}
\put(37,19){\small $0$}
\put(58.6842,31){\line(0,1){8}}
\put(49.3421,33){\line(0,1){4}}
\put(55.6842,19){\small $1$}
\put(77.3684,31){\line(0,1){8}}
\put(68.0263,33){\line(0,1){4}}
\put(74.3684,19){\small $2$}
\put(96.0526,31){\line(0,1){8}}
\put(86.7105,33){\line(0,1){4}}
\put(93.0526,19){\small $3$}
\put(114.7368,31){\line(0,1){8}}
\put(105.3947,33){\line(0,1){4}}
\put(111.7368,19){\small $4$}
\put(133.4211,31){\line(0,1){8}}
\put(124.0789,33){\line(0,1){4}}
\put(130.4211,19){\small $5$}
\put(152.1053,31){\line(0,1){8}}
\put(142.7632,33){\line(0,1){4}}
\put(149.1053,19){\small $6$}
\put(170.7895,31){\line(0,1){8}}
\put(161.4474,33){\line(0,1){4}}
\put(167.7895,19){\small $7$}
\put(189.4737,31){\line(0,1){8}}
\put(180.1316,33){\line(0,1){4}}
\put(186.4737,19){\small $8$}
\put(208.1579,31){\line(0,1){8}}
\put(198.8158,33){\line(0,1){4}}
\put(205.1579,19){\small $9$}
\put(226.8421,31){\line(0,1){8}}
\put(217.5,33){\line(0,1){4}}
\put(220.8421,19){\small $10$}
\put(245.5263,31){\line(0,1){8}}
\put(236.1842,33){\line(0,1){4}}
\put(239.5263,19){\small $11$}
\put(264.2105,31){\line(0,1){8}}
\put(254.8684,33){\line(0,1){4}}
\put(258.2105,19){\small $12$}
\put(282.8947,31){\line(0,1){8}}
\put(273.5526,33){\line(0,1){4}}
\put(276.8947,19){\small $13$}
\put(301.5789,31){\line(0,1){8}}
\put(292.2368,33){\line(0,1){4}}
\put(295.5789,19){\small $14$}
\put(320.2632,31){\line(0,1){8}}
\put(310.9211,33){\line(0,1){4}}
\put(314.2632,19){\small $15$}
\put(338.9474,31){\line(0,1){8}}
\put(329.6053,33){\line(0,1){4}}
\put(332.9474,19){\small $16$}
\put(357.6316,31){\line(0,1){8}}
\put(348.2895,33){\line(0,1){4}}
\put(351.6316,19){\small $17$}
\put(376.3158,31){\line(0,1){8}}
\put(366.9737,33){\line(0,1){4}}
\put(370.3158,19){\small $18$}
\put(395,31){\line(0,1){8}}
\put(385.6579,33){\line(0,1){4}}
\put(389,19){\small $19$}
\put(405,19){$x$}
\put(27,33){\footnotesize $0$}
\put(36,49){\line(1,0){8}}
\put(38,42){\line(1,0){4}}
\put(27,47){\footnotesize $8$}
\put(36,63){\line(1,0){8}}
\put(38,56){\line(1,0){4}}
\put(23,61){\footnotesize $16$}
\put(36,77){\line(1,0){8}}
\put(38,70){\line(1,0){4}}
\put(23,75){\footnotesize $24$}
\put(36,91){\line(1,0){8}}
\put(38,84){\line(1,0){4}}
\put(23,89){\footnotesize $32$}
\put(36,105){\line(1,0){8}}
\put(38,98){\line(1,0){4}}
\put(23,103){\footnotesize $40$}
\put(36,119){\line(1,0){8}}
\put(38,112){\line(1,0){4}}
\put(23,117){\footnotesize $48$}
\put(36,133){\line(1,0){8}}
\put(38,126){\line(1,0){4}}
\put(23,131){\footnotesize $56$}
\put(36,147){\line(1,0){8}}
\put(38,140){\line(1,0){4}}
\put(23,145){\footnotesize $64$}
\put(36,161){\line(1,0){8}}
\put(38,154){\line(1,0){4}}
\put(23,159){\footnotesize $72$}
\put(36,175){\line(1,0){8}}
\put(38,168){\line(1,0){4}}
\put(23,173){\footnotesize $80$}
\put(36,189){\line(1,0){8}}
\put(38,182){\line(1,0){4}}
\put(23,187){\footnotesize $88$}
\put(36,203){\line(1,0){8}}
\put(38,196){\line(1,0){4}}
\put(23,201){\footnotesize $96$}
\put(36,210){\line(1,0){8}}
\put(19,208){\footnotesize $100$}
\color{black}
\put(40,40.312){\circle*{4}}
\put(58.6842,40.5395){\circle*{4}}
\put(77.3684,40.4008){\circle*{4}}
\put(96.0526,40.3487){\circle*{4}}
\put(114.7368,40.2157){\circle*{4}}
\put(133.4211,40.1856){\circle*{4}}
\put(152.1053,39.9346){\circle*{4}}
\put(170.7895,40.4668){\circle*{4}}
\put(189.4737,40.3508){\circle*{4}}
\put(208.1579,40.2199){\circle*{4}}
\put(226.8421,40.5121){\circle*{4}}
\put(245.5263,41.2785){\circle*{4}}
\put(264.2105,42.0518){\circle*{4}}
\put(282.8947,46.6342){\circle*{4}}
\put(301.5789,52.2007){\circle*{4}}
\put(320.2632,56.2165){\circle*{4}}
\put(338.9474,87.7772){\circle*{4}}
\put(357.6316,157.8929){\circle*{4}}
\put(376.3158,192.0144){\circle*{4}}
\put(395,210){\circle*{4}}
\dottedline[\circle*{1}](40,37.28)(328.4181,48.2642)(363.4149,179.5959)(395,210)
\put(165,4){\small Concentration ($2^{8-x} \cdot \text{mg/l}$)}
\put(5,60){\rotatebox{90}{\small Proportion of viable bacteria (\%)}}
\end{picture}
} 
\end{center}
\caption{The best approximation $s^{\ast}$ from Example \ref{medBsp4}.}
\label{fig11}
\end{figure}
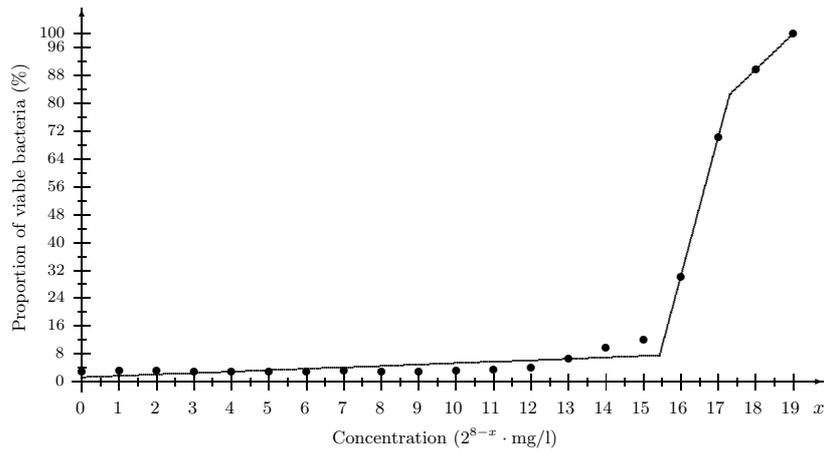



\newpage
\section*{Appendix A: Theorem 12 from \cite{cromme03}}
\addcontentsline{toc}{section}{Anhang A: Satz 12 aus \cite{cromme03}}
\setcounter{equation}{11}
\renewcommand{\theequation}{\arabic{equation}}
\begin{theorem}\label{satzEigenschaftenBesteApprox} 
Let $\mu \geq k+1 \geq 2$ and $1 \leq p \leq \infty$. For given data $(x_0,f_0),\ldots$, $(x_{\mu+1},f_{\mu+1})$ with $x_0 < x_1 < \ldots < x_{\mu} < x_{\mu+1}$ and $f_0,\ldots,f_{\mu+1} \in \mathbb{R}$ exists a best approximation $s^{\ast} \in S^1_k\left[x_0,x_{\mu+1}\right]$ 
\begin{equation*}
\big\|f-s^\ast\big\|_{p,X} \;\;=\;\; \inf_{s\in S^1_k[x_0,x_{\mu+1}]}\big\|f-s\big\|_{p,X}\;,
\end{equation*}
with the following additional features where $t_1,\ldots,t_k$ with $t_0 := x_0 < t_1 < \ldots < t_k < t_{k+1} := x_{\mu+1}$ denote the knots of $s^\ast$:
\begin{itemize}
  \item[(a)] There are no knots in the boundary regions. More precisely: 
\begin{equation*}
x_0 < x_1 \leq t_1 < \ldots < t_k \leq x_{\mu} < x_{\mu+1} \quad.
\end{equation*}
	
	\item[(b)] Data abscissae neighboring to interior knots are not knots.
	
	\item[(c)] Between two (not necessarily neighboring) interior knots of $s^\ast$ lie at least two data abscissae which are not knots of $s^\ast$.
	
	\item[(d)] On or between neighboring knots lie at least two data abscissae:
\begin{equation*}
\forall\, j,\; 0 \leq j \leq r: \;\; \exists\, i,\; 0\leq i \leq \mu: \;\; t_j \leq x_i < x_{i+1} \leq t_{j+1} \quad.
\end{equation*}

	\item[(e)] If an interior knot is situated between two neighboring data abscissae, then no additional knot lies on or between these data abscissae. That is, the proposition "$t_j$ is a knot of $s^\ast$ with $x_q < t_j < x_{q+1}$" implies $t_{j-1} < x_q$ and $t_{j+1} > x_{q+1}$.
	
	\item[(f)] Let $t_j$ be an interior knot. Then in each of the intervals $(-\infty,t_j)$ and $(t_j,\infty)$ there is a data abscissa from $\big\{x_1,\ldots,x_\mu\big\}$ which is not a knot.
	
	\item[(g)] For $p < \infty$ we have: Between an interior and a neighboring data knot of $s^\ast$ lies either exactly one data abscissa $x_q$ which is then reproduced ($t_j < x_q < t_{j+1}$, $s^\ast(x_q) = f_q$) or there exist at least two data abscissae $x_q, x_{q+1}$ between the knots ($t_j < x_q < x_{q+1} < t_{j+1}$).
	
	\item[(h)] All interior knots are proper knots, that is, the first derivative is discontinuous in all interior knots.
\end{itemize}
\end{theorem}

\end{document}